\documentclass[11pt]{amsart}
\usepackage{amssymb,amsmath,amsthm,amsfonts,amsopn,url,
mathtools,microtype,MnSymbol,
mathrsfs, tikz, tikz-cd,stmaryrd,comment,cancel}
\usepackage[pagebackref]{hyperref}
\usepackage{scalerel}
\usepackage{stackengine,wasysym}
\usepackage[shortlabels]{enumitem}

\usepackage[normalem]{ulem}
\usepackage[all]{xy}
\input xy
\xyoption{all}
\usepackage{amscd}
\usepackage{soul}

\theoremstyle{plain}
\newtheorem{thm}{Theorem}[section]

\newtheorem{prop}[thm]{Proposition}
\newtheorem{lemma}[thm]{Lemma}
\newtheorem{cor}[thm]{Corollary}

\theoremstyle{definition}

\newtheorem{defn}[thm]{Definition}
\newtheorem*{defn*}{Definition}
\newtheorem*{question*}{Question}

\newtheorem{problem}[thm]{Problem}
\newtheorem{example}[thm]{Example}
\newtheorem*{example*}{Example}
\newtheorem{rem}[thm]{Remark}
\newtheorem*{rem*}{Remark}
\newtheorem{construction}[thm]{Construction}

\newtheorem*{nota*}{Notation}

\newcommand{\field}[1]{\mathbb{#1}}
\newcommand{\N}{\field{N}}
\newcommand{\Z}{\field{Z}}
\newcommand{\Q}{\field{Q}}

\newcommand{\F}{\field{F}}

\newcommand{\A}{\field{A}}

\newcommand{\ideal}[1]{\mathfrak{#1}}

\newcommand{\p}{\ideal{p}}
\newcommand{\q}{\ideal{q}}

\newcommand{\func}[1]{\mathrm{#1} \,}

\DeclareMathOperator{\Ass}{Ass}

\newcommand{\im}{\func{im}}

\newcommand{\ra}{\rightarrow}

\DeclareMathOperator{\ann}{ann}

\newcommand{\be}{\begin{enumerate}}
\newcommand{\ee}{\end{enumerate}}

\newcommand{\li}
 {\leftfootline}

\newcommand{\onto}{\twoheadrightarrow}

\newcommand{\into}{\hookrightarrow}

\renewcommand{\phi}{\varphi}

\let\int\relax
\DeclareMathOperator{\int}{i}

\DeclareMathOperator{\End}{End}

\DeclareMathOperator{\id}{id}

\newcommand{\reg}{\mathrm{reg}}
\DeclareMathOperator{\spn}{span}

\author[Jesse Elliott]{Jesse Elliott{$^*$}}
\address{Department of Mathematics\\
California State University, Channel Islands\\
Camarillo, CA\\ USA
}
\thanks{{$^*$}Jesse Elliott passed away June 29, 2025, after acceptance but before publication of this article. Jesse was a crucially valued member of the commutative ring theory community, typically finding the right generalization of a concept before anyone else, and will be greatly missed.}

\author{Neil Epstein}
\address{Department of Mathematical Sciences \\ George Mason University \\ Fairfax, VA \\
USA}
\email{nepstei2@gmu.edu}

\title[Factroids]{Additive subgroups of a module that are saturated with respect to a subset of the ring}

\date{August 1, 2025}

\begin{document}

\dedicatory{Dedicated to K. Alan Loper on the occasion of his retirement}
\begin{abstract}
Let $T$ be a subset of a ring $A$, and let $M$ be an $A$-module.  We study the additive subgroups $F$ of $M$ such that, for all $x \in M$, if $tx \in F$ for some $t \in T$, then $x \in F$.  We call any such subset $F$ of $M$ a \emph{$T$-factroid of $M$}, which is a kind of dual to the notion of a $T$-submodule of $M$.  We connect the notion with the zero-divisors on $M$, various classes of primary and prime ideals of $A$, Euclidean domains, and the recent concepts of unit-additive commutative rings and of Egyptian fractions with respect to a multiplicative subset of a commutative ring.  We also introduce a common generalization of local rings and unit-additive rings, called \emph{sublocalizable} rings, and relate them to $T$-factroids.
\end{abstract}
\maketitle
\tableofcontents

\section{Introduction}

Throughout this paper, $A$ denotes a ring, $M$ denotes a (left) $A$-module, and $W \subseteq A$ denotes a multiplicatively closed set.   All rings are unital unless otherwise stated. 

We introduce here a new type of algebraic structure  that is motivated on several grounds.  Specifically, for a given subset $T$ of $A$, we consider the additive subgroups $F$ of $M$ such that, for all $x \in M$, if $tx \in F$ for some $t \in T$, then $x \in F$.  We call any such subset $F$ of $M$ a \emph{$T$-factroid of $M$}.  Its obvious dual notion is that of a \emph{$T$-submodule of $M$}, which is an additive subgroup $F$ of $M$ such that $tx \in F$ whenever $x \in F$ and $t \in T$.  Since a $T$-submodule of $M$ is equivalently a $\Z\langle T \rangle$-submodule of $M$ by restriction of scalars, where $\Z\langle T \rangle$ is the subring of $A$ generated by $T$, the notion of a $T$-submodule can be reduced to that of a $B$-submodule for the subrings $B$ of $A$.  The notion of a $T$-factroid, however, appears to be new, and it is thus motivated abstractly as a dual to the notion of a $T$-submodule.

However, the initial motivation for our exploration of $T$-factroids was our discovery of their intimate connection the study of Egyptian domains, the reciprocal complement of an integral domain, and their respective generalizations to commutative rings with zerodivisors.  Recall \cite{GLO-Egypt} that for an
integral domain $A$ with fraction field $K$, an element $\alpha \in K$ is \emph{$A$-Egyptian} if there exist distinct nonzero elements $d_1, \ldots, d_n$ of $A$ such that $\alpha = \sum_{i=1}^n \frac 1{d_i}$, so-called because of Egyptian fractions in number theory \cite{BlEl-Egypt} (and ultimately in medieval European \cite{DuGr-FibE} and ancient Egyptian 
\cite{Rei-count} mathematics).  It follows from \cite[Theorem 1]{GLO-Egypt} and \cite[Remark 2.2]{nme-Euclidean} that the $A$-Egyptian elements of $K$ form a subring $R(A)$ of $K$, called the \emph{reciprocal complement} of $A$.    An integral domain $A$ is said to be \emph{Egyptian} if every element of $A$, or, equivalently, of its quotient field $K$, is $A$-Egyptian \cite{GLO-Egypt}.  Clearly, then, a domain $A$ is Egyptian if and only if its reciprocal complement is as large as it could be, namely $K$. Although some reciprocal complements are easy to understand (e.g. classically $R(\Z) = \Q$ (see \cite{VanVan-rbaseZ} and \cite{DuGr-FibE}), and the reciprocal complement of a Euclidean domain is a DVR or a field \cite[Main Theorem]{nme-Euclidean}), some are not (e.g. the reciprocal complement of $k[x,y]$, for $k$ a field, is objectively very strange \cite{nmeGuLo}). For more on Egyptian domains and reciprocal complements, see also \cite{nme-Edom, Gu-recomp, nme-Emult, nme-ratfuns}.

It is trivial to see that $R(A)$, for any integral domain $A$, is an $(A\setminus \{0\})$-factroid of its quotient field  $K$.  In fact, the connection between the $R(A)$ and  the $(A\setminus \{0\})$-factroids is deeper than that.  Note first that the  $T$-factroids of $M$ are closed under arbitrary intersections (see Proposition~\ref{pr:intersect}), so  one can define the \emph{$T$-factroid of $M$ generated by a set $S$}, denoted $[S]^T_M$, to be the smallest $T$-factroid of $M$ containing $S$.  One then has the following:

\begin{thm}[see Theorem~\ref{thm:recipG}]\label{thm:initialegypt2}
Let $A$ be an integral domain, $K$ its fraction field, $0 \neq \alpha \in K$, and $T = A \setminus \{0\}$.  Then the following are equivalent: \begin{enumerate}[(a)]
    \item $\alpha \in R(A)$.
    \item There exist $f,g \in A$ with $\alpha = f/g$ and $f \in [\{g\}]^T_A$.
    \item For any $f,g \in A$ with $\alpha = f/g$, there is some nonzero $h\in A$ with $hf \in [\{hg\}]^T_A$.
    \end{enumerate}
\end{thm}

Of course, Theorem \ref{thm:initialegypt2} begs to be generalized to the $T$-factroids of a general $A$-module $M$ for any arbitrary subset $T$ of $A$.  In Section~\ref{sec:Egyptian}, this is achieved in the case where $A$ is any commutative ring.

In Section \ref{sec:Euclidean},  we use factroids to analyze Euclidean domains.  Namely, we classify the factroids of Euclidean domains in terms of the values of their Euclidean functions, and show that one version of the classification holds if and only if the domain is of the form $k[x]$.

 A commutative ring is said to be \emph{unit-additive} if a sum of units  is either a unit or nilpotent (see \cite{nmeSh-unitadd}). 
 We define the following more general class of rings: a commutative ring $A$ is \emph{sublocalizable} if a sum of units is  either a unit or an element of its Jacobson radical $J(A)$. 
Such a ring always has a distinguished local subring consisting of its units along with $J(A)$, which we call its \emph{sublocalization}.  We show that the sublocalization of a sublocalizable ring is a  particularly nice factroid of $A$, namely, the subring $\Z \langle A^\times \rangle$ of $A$ generated by the group of units of $A$. In Section~\ref{sec:unitadditive}, we provide strong connections between $T$-factroids,  unit-additive rings, and sublocalizable  rings.   

Examples of $T$-factroids are  ubiquitous in  both commutative and noncommutative algebra.  Indeed, for any additive subgroup $F$ of $M$, it is clear that \[ W_M(F) := W(F) := \{a \in A \mid (F:_M a) \subseteq F\}\]
is a subset $T$ of $A$ for which $F$ is a $T$-factroid of $M$. In fact, 
by its definition 
it is the largest such subset, and one can  easily check that it is also  a multiplicative submonoid of $A$.  In many instances, one can express $W_M(F)$ in terms of known constructs, and vice versa.  For example, if $F$ is an $A$-submodule of $M$, then $W_M(F)$ is just the saturated multiplicative monoid consisting of all elements of $A$ that are regular on the $A$-module $M/F$ (see Proposition \ref{pr:modulefactroid}). Thus, for example $W_{{}_AA}(\{0\})$ is just the monoid of all left-regular elements of $A$, where ${}_AA$ denotes $A$ as a left $A$-module. From the theory of primary decomposition, then, it follows that, if $M$ is a finitely generated module over a commutative Noetherian ring $A$, then $W_M(0)$  
is equal to 
the
complement of the 
union of the associated primes of $M$  (see Proposition~\ref{pr:pdec}).
Moreover, if $A$ is an arbitrary commutative ring, then one has $W_A(\sqrt{I}) = A \setminus \bigcup \operatorname{Min}(I)$  for any ideal $I$ of $A$, where $\operatorname{Min}(I)$ is the set of all prime ideals of $A$ that are minimal over $I$  (see Proposition \ref{prop:Acircrad}).   Thus, $I$ is prime if and only if $W_A(I) = A\setminus I$.   In particular, $W(\sqrt{0})$ is the saturated multiplicative set $A^\circ$ of all elements of $A$ that do not lie in any minimal prime of $A$, a set that figures centrally in the definition of tight closure over a commutative Noetherian ring of prime characteristic \cite{HHmain}.  Also, it is easy to check that, again if $A$ is commutative, then any intersection $\bigcap  \q_i$ of primary ideals $\q_i$  of $A$ is an $(A \setminus \bigcup \sqrt{\q_i})$-factroid of $A$  (see Proposition \ref{prop:intprimeary}). 

Further instances of $T$-factroids abound.  For example, if $A = k[x]$ for some field $k$,  then the $(A\setminus \{0\})$-factroids of $A$ besides $\{0\}$ and $A$ are precisely the $k$-spans of $\{1,x,x^2,\ldots,x^m\}$ for all nonnegative integers $m$.  (See  Theorem  \ref{prop:normeuclidean} for a generalization of this fact to more general Euclidean domains.) However, if $A = k[x_1, \ldots, x_n]$ for $n\geq 2$, then there are more $(A \setminus \{0\})$-factroids than just the $k$-spans of $\{1,x_i,x_i^2,\ldots,x_i^m\}$ for $m \geq 0$ and $1 \leq i \leq n$. For instance, if $f\in A$ such that $f+c$ is an irreducible polynomial for all $c\in k$, then $\spn_k\{1,f\}$ is an $(A \setminus \{0\})$-factroid (see Example~\ref{ex:kxy}). The problem of determining the $(A\setminus \{0\})$-factroids of $A$ remains open, even when $n=2$. 

Our paper is structured as follows. Section~\ref{sec:firststuff} introduces most of the main definitions and basic results. Section~\ref{sec:dual} develops duality and interaction of factroid theory with module theory; at the end of the section, we include an extension of the theory to the context of abelian group endomorphisms.  In Section~\ref{sec:rings}, we provide the reader with examples of factroids of rings for various familiar classes and constructions of commutative rings.  In Section~\ref{sec:generation}, we explore what it means for a set to \emph{generate} a $W$-factroid of a module or a ring.  In Section~\ref{sec:regular}, we develop the theory of $W$-\emph{regular} $T$-factroids,  the necessary theoretical foundation for Section~\ref{sec:Egyptian}, where we make the connection with Egyptian domains and related topics. In Section~\ref{sec:Euclidean}, we show that the factroids of a Euclidean domain have a particularly elegant description in terms of Euclidean functions.  Section~\ref{sec:unitadditive} is where we discuss unit-additive and sublocalizable rings.  In Section~\ref{sec:maps}, we discuss how the objects and operations of factroid theory work when passing along ring and module homomorphisms.  Finally, in Section~\ref{sec:graded}, we show that in certain common graded settings, finiteness conditions apply that can make it easier to determine explicitly the membership in factroids generated by a given set.

We thus hope to have shown the use and appeal of $T$-factroids in commutative and noncommutative algebra.

\section{$T$-factroids of a module}
\label{sec:firststuff}

 In this section, we define and prove some basic facts about the {\it $T$-factroids of $M$}.  The no(ta)tions developed here will be crucial for all that follows.

\begin{defn}\label{def:firstdefs}
  Let $A$ be a ring, $I$ a (two-sided) ideal of $A$, and $M$ a (left) $A$-module. Let $T \subseteq A$ and $S,U \subseteq M$. Throughout the paper, we use the following notation and conventions: \begin{itemize}
        \item $Z(A) := $ the center of $A$.
        \item $A^\times :=$ the group of units of $A$.
        \item $\reg(A) := $ the multiplicative monoid of all left-regular elements, i.e.,  all left nonzerodivisors, of $A$.
         \item $\langle T \rangle$ denotes the multiplicative submonoid of $A$ generated by $T$.
        \item  $\Z\langle T \rangle$ denotes the subring of $A$ generated by $T$. \begin{itemize}
            \item Thus, $\Z\langle T \rangle =\Z\langle \langle T \rangle \rangle$ is the set of all sums of elements of $\langle T \rangle$.
        \end{itemize}
        \item $\widetilde{T} = \{a \in A: ba \in T \text{ for some } b \in A\}$ denotes the \emph{(left) saturation} of $T$.
        \item One says that $T$ is \emph{(left) saturated} if $\widetilde{T} = T$, that is, if $ba \in T$ implies $a \in T$ for any $a,b \in A$.
        \item $T^{\operatorname{sm}}$ denotes the smallest saturated multiplicative submonoid of $A$ containing $T$. \begin{itemize}
            \item $T^{\operatorname{sm}}$  exists because saturated sets and multiplicative submonoids are closed under arbitrary intersections.
            \item Note that if $A$ is commutative, then $T^{\operatorname{sm}} = \widetilde{\langle T \rangle}$.
        \end{itemize}
        \item If $T \subseteq A^\times$, then $T^{-1} := \{t^{-1} \mid t \in T\} \subseteq A^\times.$ 
        \item $(S :_M T) := \{x \in M \mid tx \in S \text{ for all } t\in T\}$.
        \begin{itemize}
            \item When $T = \{t\}$ is singleton, write $(S :_M t) := (S:_MT)$.
        \end{itemize}
        \item $(S :_A U) := \{a \in A \mid au \in S \text{ for all } u\in U\}$.
        \begin{itemize}
            \item When $U=\{u\}$ is a singleton, write $(S :_Au) := (S :_AU)$.
        \end{itemize}
        \item $\reg_M(A) := \{a \in A \mid (0:_M a) = 0\}$ denotes the set of all \emph{nonzerodisivors on $M$} (also known as \emph{$M$-regular elements}).
        \begin{itemize}
            \item Note that $\reg(A)  = \reg_{{}_AA}(A)$, where ${}_AA$ is the left $A$-module $A$.
        \end{itemize}
        \item When $A$ is commutative, $A^\circ := $ the complement of the union of the minimal (i.e., height zero) prime ideals of $A$. \begin{itemize}
            \item Note that $A^\circ$ is a saturated multiplicative submonoid of $A$.
            \item Note that $\reg(A) \subseteq A^\circ$, with equality if $A$ is reduced.
        \end{itemize}
        \item $J(A) := $ the Jacobson radical of $A$.
        \item $\sqrt{I} := $ the radical of $I$.
        \item The set $\N$ of all natural numbers includes $0$, by convention.
     \end{itemize}
\end{defn}

\begin{defn}\label{def:factroid}
    Let $T \subseteq A$. 

    We say that $S \subseteq M$ is \emph{$T$-saturated (in $M$)}  if $(S :_M t) \subseteq S$ for all $t\in T$.  In other words, $S$ is $T$-saturated if, for all $x \in M$ such that $tx \in S$ for some $t \in T$, one has $x \in S$.  In particular, a subset $T$ of $A$ is a saturated set in the sense of Definition~\ref{def:firstdefs} if and only if it is $A$-saturated in (the left module) $A$.
    
    We say that $F \subseteq M$ is a \emph{$T$-factroid of $M$} 
    if $F$ is a $T$-saturated additive subgroup of $M$.  In other words, $F$ is a \emph{$T$-factroid of $M$} if $F$ is an additive subgroup of $M$ such that, for all $x \in M$ such that $tx \in F$ for some $t \in T$, one has $x \in F$.   By a \emph{(left) $T$-factroid of $A$} we mean a $T$-factroid of the left $A$-module ${}_AA$.
\end{defn}

Colloquially, a $T$-factroid of an $A$-module is an additive subgroup of the module that absorbs colons from $T$.  

\begin{defn}
    Let $S \subseteq M$ and $T \subseteq A$.  The \emph{$T$-saturation of $S$} is the set
    $$\operatorname{Sat}^T_M(S) : = \bigcup_{w \in \langle T \rangle} (S:_M w) = \{x \in M \mid wx \in S \text{ for some } w \in \langle T \rangle\}.$$
    It is the smallest $T$-saturated subset of $M$ containing $S$.
\end{defn}

\begin{example}
    One has $\widetilde{T} = \operatorname{Sat}^A_A(T)$ for any subset $T$ of $A$.
\end{example}

\begin{rem}
If $U \subseteq T$ are subsets of $A$, then any $T$-factroid of $M$ is a $U$-factroid of $M$. 
 If $T \subseteq \{\pm 1\}$, then a $T$-factroid of $M$ is the same as an  additive subgroup of $M$.   Suppose that $-1 \in T$.  Then to check that $F \subseteq M$ is a $T$-factroid, in addition to showing that $(F:_M t) \subseteq F$ for all $t \in T$, one need only show that $F$ is nonempty and additively closed. This is because for any $x\in F$, we have $-x \in (F:_M -1)$. 
\end{rem}

\begin{example}\label{rem:basics}
It is clear that $M$ is a $T$-factroid of $M$.   However, the trivial $A$-submodule $\{0\}$ of $M$ is a $T$-factroid of $M$ if and only if every element of $T$ is a nonzerodivisor on $M$.  Moreover, if $0\in T$, or more generally if $T \cap \ann_A M  \neq \emptyset$, then $M$ is the only $T$-factroid of $M$. 
\end{example}

\begin{rem}
    Let $T \subseteq A$.  Suppose that $F$ is a $T$-factroid of $M$ and $B$ is a subring of $A$.  Then $F$ is a $(T \cap B)$-factroid of the $B$-module $M$ obtained by restriction of scalars.  Thus, if $T \subseteq B$, then $F$ is also a $T$-factroid of the $B$-module $M$.
\end{rem}

\begin{rem}\label{rem:submodfactroid}
    Let $T \subseteq A$.  A $T$-factroid of $M$ is equivalently an additive subgroup $F$ of $M$ such that $\ker(M \stackrel{t \cdot }{\to} M \to M/F) \subseteq \ker(M \to M/F)$ for all $t \in T$, where $M \to M/F$ is the canonical projection of abelian groups.
\end{rem}

\begin{prop}\label{pr:Tfact}
   Let $T \subseteq A$. A $T$-saturated subset of $M$ is equivalently a $\langle T \rangle$-saturated subset of $M$. Similarly, a $T$-factroid of $M$ is equivalently a $\langle T \cup \{\pm 1 \} \rangle$-factroid of $M$.   
\end{prop}

\begin{proof}
    Let $S$ be any subset of $M$.  Clearly, for any $a,b \in A$,  if $(S:_M a) \subseteq S$ and $(S:_M b) \subseteq S$, then  $(S:_M ab) = ((S :_M a) :_M b)) \subseteq S$.  Moreover, one has $(S :_M 1) = S$,  and $(S :_M -1) = S$ if $S = -S$.
\end{proof}

 Thus, one may always assume without loss of generality that $T$ is multiplicatively closed, or even a submonoid of $A$ containing $-1$.

\begin{prop}\label{pr:intersect}
  Let $T \subseteq A$. The intersection of any collection of $T$-saturated subsets of $M$ is a $T$-saturated subset of $M$.    Likewise, the intersection of any collection of $T$-factroids of $M$ is a $T$-factroid of $M$. 
\end{prop}

\begin{proof}
Let $\{S_\alpha \mid \alpha \in \Lambda\}$ be a collection of $T$-saturated subsets of $M$, and let $S = \bigcap_{\alpha \in \Lambda} S_\alpha$.  Let $t\in T$ and $x \in M$ with $tx\in S$.  Then one has $tx \in S_\alpha$, whence $x\in S_\alpha$, for all $\alpha \in \Lambda$, and therefore $x\in S$.  Moreover, if each $S_\alpha$ is an additive subgroup of $M$, then so is $S$.
\end{proof}

In light of Proposition \ref{pr:intersect}, we can make the following definition.

\begin{defn}
  Let $T \subseteq A$.  For any subset $S$ of $M$, the \emph{$T$-factroid of $M$ generated by $S$} is the smallest $T$-factroid $M$ containing $S$, denoted $[S]_M^T $.  In particular, \[
    [S]_M^T = \bigcap \{F \mid F \text{ is a }T\text{-factroid of }M\text{ and }S \subseteq F\}.
    \]
  When $M$ is understood, we abbreviate $[S]_M^T$ to $[S]^{T}$.  A $T$-factroid generated by a singleton is \emph{principal}. In this case, when $S = \{x\}$ for some $x\in M$, we write $[x]_M^T$ or $[x]^T$ in place of $[\{x\}]^T_M$.
\end{defn}

By Proposition \ref{pr:Tfact}, we have the following.

\begin{cor}
For all $S \subseteq M$ and all $T \subseteq A$,  $[S]_M^T = [S]_M^{\langle T \rangle} = [S]_M^{\langle T \cup \{\pm 1\} \rangle}.$
\end{cor}

\begin{lemma}\label{lem:noncommempty}
Let $S\subseteq M$ and $T \subseteq A$.  One has \[
    [S]^T_M \supseteq \operatorname{Sat}^T_M(S) =  \bigcup_{w \in \langle T \rangle} (S :_M w) = \{x\in M \mid wx \in S \text{ for some } w\in \langle T\rangle \}.\] 
\end{lemma}

\begin{proof}
    This follows because any $T$-factroid of $M$ is $T$-saturated in $M$.
\end{proof}

\begin{prop}\label{prop:noncommempty}
Let $T \subseteq A$.  The $T$-factroid $[\emptyset]^T_M =[0]^T_M$ of $M$ is the smallest $T$-factroid of $M$.  
    Moreover, one has \[
    [\emptyset]^T_M \supseteq \operatorname{Sat}^T_M(\{0\}) =  \bigcup_{w \in \langle T \rangle} (0 :_M w) = \{x\in M \mid wx=0 \text{ for some } w\in \langle T\rangle \}.\] 
    Equality holds if $T \subseteq Z(A)$, or more generally if $st=ts$ for every $t,s \in T$.
\end{prop}

\begin{proof}
By Lemma~\ref{lem:noncommempty}, we need only prove the last statement of the proposition.  We may suppose without loss of generality that $W: = T$ is a  multiplicative submonoid of $A$.
 It suffices to show that the $T$-saturation $F$ of $\{0\}$ in $M$ is an additive subgroup of $M$.  To see this, let $x, y \in F$, so that $wx = 0$ and $vy = 0$ for some $w,v \in W$.  Then $vw(x-y) = vwx - wvy = 0$, since $vw = wv$.  Finally, since $vw \in W$, it follows that  $x-y \in F$.
\end{proof}

\begin{cor}
If the ring $A$ is commutative, then for any multiplicative set $W$,  $$[\emptyset]_M^W  = \ker(M \ra W^{-1}M),$$ where $M \ra W^{-1}M$ is the localization map.
\end{cor}

\begin{example}\label{ex:0notaddcl}
         If the elements of $T\subseteq A$ do not commute with one another, then the   union in Proposition \ref{prop:noncommempty} need not be additively closed.

     For instance, let $A = k\langle x,y\rangle$ (a ring freely generated over a field $k$ by two non-commuting indeterminates), $W = T = \langle x,y\rangle$, and $M = (A \oplus A)/(xA \oplus yA)$.  Then in $M$, we have that $x(1,0)= (\bar x, 0)= 0$ and $y(0,1) = (0,\bar y) = 0$, but $\ann_A(1,1) = 0$.  Thus, $(1,0), (0,1) \in \bigcup_{w\in W} (0 :_Mw)$, but their sum is not.
\end{example}

\begin{defn}\label{def:WofF}
Given a subset $S$ of $M$, we let \[ W_M(S) := W(S) := \{a \in A \mid (S:_M a) \subseteq S\}.\]
\end{defn}

\begin{prop}
    Let $S$ be any subset of $M$.  Then $W(S)$ is a multiplicative submonoid of $A$.  Moreover, $W(S)$ is the largest subset $T$ of $A$ for which $S$ is $T$-saturated in $M$; in other words, for any subset $T$ of $A$, $S$ is $T$-saturated in $M$ if and only if $T \subseteq W(S)$. \end{prop}

\begin{proof}
    This is clear from the definitions and from Proposition \ref{pr:Tfact}.
\end{proof}

\begin{cor}
    Let $F$ be any additive subgroup of $M$.  Then $W(F)$ is a multiplicative submonoid of $A$ containing $-1$.  Moreover, $W(F)$ is the largest subset $T$ of $A$ for which $F$ is a $T$-factroid of $M$; in other words, for any subset $T$ of $A$, $F$ is a $T$-factroid of $M$ if and only if $T \subseteq W(F)$. \end{cor}

\begin{rem}
An additive subgroup $F$ of $M$ is proper if and only if $0 \notin W(F)$, if and only if $W(F) \subsetneq A$.  Moreover, $W(F)$ is saturated if and only if, for all $T \subseteq A$,  $F$ is a $T$-factroid of $M$ if and only if $F$ is a $T^{\operatorname{sm}}$-factroid of $M$.

The fact that the intersection of any collection of $T$-factroids of $M$ is a $T$-factroid of $M$ for any subset $T$ of $A$ is equivalent to the inclusion $W_M(\bigcap_{\lambda \in \Lambda} F_\lambda) \supseteq \bigcap_{\lambda \in \Lambda} W_M(F_\lambda)$ for any collection $\{F_\lambda: \lambda \in \Lambda\}$ of additive subgroups of $M$.  Indeed, each $F_\lambda$ is a $T$-factroid of $M$ if and only if $T \subseteq W(F_\lambda)$ for all $\lambda$, if and only if $T \subseteq \bigcap_{\lambda \in \Lambda} W(F_\lambda)$; and $\bigcap_{\lambda \in \Lambda} F_\lambda$ is a $T$-factroid of $M$ if and only if $T \subseteq W (\bigcap_{\lambda \in \Lambda} F_\lambda)$.
\end{rem}

\begin{prop}\label{prop:W0}
   One has $W_M(\{0\}) = \reg_M(A)$. Equivalently, the saturated submonoid $\reg_M(A)$ of $A$ is the largest subset $T$ of $A$ such that $\{0\}$ is a $T$-factroid of $M$. 
\end{prop}

\begin{proof}
    This is immediate from the definitions.
\end{proof}

\begin{cor}\label{cor:W0}
     One has  $W_A((0)) = \reg(A)$.  Equivalently, the saturated submonoid $\reg(A)$ of $A$ is the largest subset $T$ of $A$ such that the zero ideal $(0)$ of $A$ is a $T$-factroid of $A$.
\end{cor}

\section{The duality between factroids and submodules}\label{sec:dual}

If $A$ is an integral domain, so that $\reg(A) = A^\circ = A \setminus\{ 0\}$, then the notion of a $\reg(A)$-factroid of $M$ is in a sense dual to the notion of an $A$-submodule of $M$.  Indeed, it is clear that an additive subgroup $L$ of $M$ is an $A$-submodule of $M$ if and only if $L \subseteq (L :_M a)$ for every $a \in \reg(A)$, while the notion of a $\reg(A)$-factroid of $M$ has the given inclusion reversed.

We can make the above duality more explicit and general as follows:

\begin{defn}\label{def:Tsubmodule}
Let $T \subseteq A$.  We say that a \emph{$T$-submodule of $M$} is an additive subgroup $F$ of $M$ such that $TF \subseteq F$, or equivalently such that $F \subseteq (F :_M t)$ for all $t \in T$.
\end{defn}

In this section, we explore the duality between $T$-submodules and $T$-factroids.  The correspondence is especially strong when $T$ consists of units.  We also give special attention to the theory of regular elements and primary decomposition.  We conclude the section with an extension of factroid theory to the setting of endomorphisms of abelian groups.

\begin{rem} If $B$ is a subring of $A$, then a $B$-submodule of $M$ by restriction of scalars is equivalently a $B$-submodule of $M$ in the sense of Definition \ref{def:Tsubmodule}.
\end{rem}

\begin{defn}\label{def:AofF} For any additive subgroup $F$ of $M$, we let
$$A(F) :=  (F :_A F) =  \{ a \in F: aF \subseteq  F\}  = \{a \in A:  F \subseteq (F :_M a)\}.$$ 
\end{defn}

\begin{prop} Let $F$ be an additive subgroup of $M$.  Then the following statements hold: \begin{enumerate}
    \item $A(F)$ is the largest subset $T$ of $A$ for which $F$ is a $T$-submodule of $M$.
    \item For any subset $T$ of $A$, $F$  is a $T$-submodule of $M$ if and only if $T \subseteq A(F)$.
    \item $A(F)$ is a subring of $A$.  Indeed, it is the largest subring $B$ of $A$ for which $F$ is a $B$-submodule of $M$.
    \item $F$ is an $A$-submodule of $M$ if and only if $A(F) = A$.
    \item For any subset $T$ of $A$, $F$ is a $T$-submodule of $M$ if and only if $F$ is a $\Z\langle T \rangle$-submodule of $M$. 
\end{enumerate}
\end{prop}
   
\begin{proof}
    All statements follow immediate from the definitions.
\end{proof}

\begin{rem}
    Let $\operatorname{Mod}_\Z(M)$ denote the (complete) poset of all additive subgroups of $M$.  Then $C_a : F \mapsto (F :_M a)$ is an operation on the set $\operatorname{Mod}_\Z(M)$, for any $a \in A$. Moreover, the association $a \mapsto C_a$ defines a map $A \to \operatorname{End}_{\mathbf{Set}}(\operatorname{Mod}_\Z(M))$ with $C_1 = \id$ and $C_b \circ C_a = C_{ab}$ for all $a,b \in A$.  For all $F \in \operatorname{Mod}_\Z(M)$, one has
    $W(F) = \{a \in A: C_a(F) \subseteq F\}$
    and
    $A(F)  = \{a \in A: F \subseteq C_a(F)\}$. 
    \end{rem}

\begin{prop}\label{prop:unitfactroid}
    Let $T$ be a subset of $A^\times$.  Then $\Z\langle T^{-1} \rangle = [1]^T_A$, and a $T$-factroid of $M$ is equivalently a $T^{-1}$-submodule of $M$, or equivalently a $\Z\langle T^{-1} \rangle$-submodule of $M$.  Conversely, $\Z\langle T \rangle = [1]^{T^{-1}}_A$, and a $T$-submodule of $M$ is equivalently a $\Z\langle T \rangle$-submodule of $M$, or equivalently a $T^{-1}$-factroid of $M$.
\end{prop}

\begin{proof}
    This follows from the fact that when $w\in A^\times$, $(F :_M w)= w^{-1}F$, and therefore $(F:_M w) \subseteq F$ if and only if $w^{-1}F \subseteq F$, for any subset $F$ of $M$.
\end{proof}

\begin{cor}
 Suppose that $W$ is a subgroup of $A^\times$.  Then a $W$-factroid of $M$ is equivalently a $W$-submodule of $M$, or equivalently a $B$-submodule of $M$, where $B = \Z\langle W \rangle = [1]^{W}_A$ is the subring of $A$ consisting of all finite sums of elements of $W$.
\end{cor}

\begin{cor}\label{cor:duality}
    Suppose that $A$ is commutative, and let $V \subseteq W$ be multiplicative subsets of $A$.  Let $\frac{V}{1} = \{v/1: v \in V\}$ denote the image of the set $V$ in $W^{-1}A$, and let $\frac{1}{V} = \left (\frac{V}{1}\right)^{-1}$ denote the multiplicative set of reciprocals of the elements of $\frac{V}{1}$ in $W^{-1}A$.  Then a $\frac{V}{1}$-factroid of $W^{-1}M$ is equivalently a $\frac{1}{V}$-submodule of $W^{-1}M$, or equivalently a $\Z\langle \frac{1}{V} \rangle$-submodule of $W^{-1}M$.  Conversely,  a $\frac{V}{1}$-submodule of $W^{-1}M$ is equivalently a $\Z\langle \frac{V}{1} \rangle$-submodule of $W^{-1}M$, or equivalently a $\frac{1}{V}$-factroid of $W^{-1}M$.   Moreover, one has  $\Z\langle\frac{1}{V}\rangle  = [1]^{\frac{V}{1}}_{W^{-1}M}$ and $\Z\langle\frac{V}{1}\rangle  = [1]^{\frac{1}{V}}_{W^{-1}M}$.
\end{cor}

\begin{prop}\label{prop:Wunits}
Let $F$ be any additive subgroup of $M$.
Then the units of $A$ in $W(F)$ are precisely the inverses of the units of $A$ in $A(F)$, that is,
$$W(F) \cap A^\times = (A(F) \cap A^\times)^{-1}.$$
  Thus, the units of $A$ in $W(F)$ are completely determined by the ring $A(F)$.
Moreover, the multiplicative monoid
$$W(F) \cap A(F) = \{a \in A \mid (F :_M a) = F\}$$
is the largest subset $T$ of $A$ for which $F$ is both a $T$-submodule and $T$-factroid of $M$.  
\end{prop}

\begin{proof}
  This follows readily from the fact that $(F :_M w) = w^{-1} F$ for all units $w$ of $A$.
\end{proof}

We next look at when $A$-submodules are $T$-factroids.

\begin{prop}\label{pr:modulefactroid} Let $L$ be an $A$-submodule of $M$.  Then  $W_M(L) = \reg_{M/L}(A)$;  therefore $W_M(L)$ is saturated.   Moreover, for any subset $T$ of $A$, the following are equivalent: \begin{enumerate}[(a)]
    \item $L$ is a $T$-factroid of $M$.
    \item $T \subseteq \reg_{M/L}(A)$.
    \item $L$ is a $T^{\operatorname{sm}}$-factroid of $M$.
\end{enumerate}  Thus, if $A$ is commutative, then $W_M(L) = \ker(M/L \ra W^{-1}(M/L))$, where $W$ is the multiplicative set $\langle T \rangle$ (or $T^{\operatorname{sm}} =  \widetilde{\langle T \rangle}$) and  $M/L \ra W^{-1}(M/L)$ is the localization map.
\end{prop}

\begin{proof}
We have the following equivalences for any 
$t \in A$: \begin{align*}
    t \in W_M(L) & \Leftrightarrow  (L :_M t) \subseteq L \\
    &\Leftrightarrow (0 :_{M/L} t) = 0\\
     & \Leftrightarrow  t \in \reg_{M/L}(A).
\end{align*}
The proposition follows.
\end{proof}

Since every additive subgroup $F$ of $M$ is an $A(F)$-submodule of $M$, so that $M/F$ is an $A(F)$-module, one can relativize Proposition \ref{pr:modulefactroid} to characterize the additive subgroups of $M$ that are both $T$-submodules of $M$ and $T$-factroids $F$ of $M$ in terms of the $A(F)$-module $M/F$, as follows.

\begin{cor}
    Let $T$ be a subset of $A$.  An additive subgroup $F$ of $M$ is both a $T$-submodule $F$ of $M$ and $T$-factroid  of $M$ if and only if $T \subseteq A(F)$ and the elements of $T$ are all nonzerodivisors on the $A(F)$-module $M/F$.
\end{cor}

Recall that a  proper submodule $L$ of a module $M$ over a commutative ring $A$ is \emph{primary} if for any $r\in A$, either $(L :_M r) = L$ or $r^n M \subseteq L$ for some $n\in \N$.

\begin{prop}\label{pr:submoduleprimary}
Suppose that $A$ is commutative. Let $L$ be proper $A$-submodule of $M$, and let $T \subseteq A$. If $L$ is a $T$-factroid of $M$, then $\langle T \rangle \cap \ann_A(M/L) = \emptyset$.  The converse holds if $L$ is a primary submodule of $M$.
\end{prop}

\begin{proof}
Without loss of generality, we may assume that $W := T$ is multiplicatively closed.  Suppose there is some $r \in W \cap \ann_A(M/L)$.  Let $x\in M \setminus L$.  Then $rx \in L$ but $x \notin L$, so $L$ is not a $W$-factroid of $M$.

On the other hand, suppose $I \cap W = \emptyset$, where $I=\ann_A(M/L)$, and $L \subseteq M$ is primary.  Let $r \in W$ and $x\in M$ with $rx \in L$.  Since $L$
is primary in $M$, either $r \in \sqrt{I}$ or $x\in L$.  But if $r\in \sqrt{I}$, then there is some $n$ with $r^n \in I \cap W$, contrary to hypothesis. Hence, $x\in L$.  Thus, $L$ is a $W$-factroid of $M$.
\end{proof}

One can identify even more precisely what is happening when primary decompositions are present.

\begin{prop}\label{pr:pdec}
Let $A$ be a commutative Noetherian ring and $L \subseteq M$ finite $A$-modules.  Then $W(L) = A \setminus \bigcup \Ass(M/L)$.  More generally, if $L\subseteq M$ admits a primary decomposition, then $W(L)$ is the complement of the union of the primes belonging to $L$ in $M$, in the sense of \cite[Exercises 4.22-23]{AtMac-ICA}.
\end{prop}

\begin{proof}
In either situation, the zerodivisors on $M/L$ are precisely the elements of the given union. Then Proposition~\ref{pr:modulefactroid} finishes the proof. 
\end{proof}

\begin{rem} 
Proposition~\ref{pr:pdec} does not always hold when $M=A$ and $L=(0)$ does not admit a primary decomposition, even if the set $\Ass(M)$ is replaced by the set of strong Krull or weak Bourbaki primes of $M$. See \cite[Example 6.4]{nmeSh-sKflat}.
\end{rem}

\begin{rem}
Although the elementary aspects of this subject apply to noncommutative rings, most of the examples and applications we found were for commutative rings.  This leaves the subject ripe for further study over noncommutative rings.
\end{rem}

\begin{rem}\label{sec:endo}
Let $G$ be an abelian group, and let $E \subseteq \operatorname{End}(G) = \operatorname{End}_\Z(G)$.  Define an {\it $E$-factroid of $G$} to be a subgroup $F$ of $G$ such that $\varphi^{-1}(F) \subseteq F$ for all $\varphi \in E$.  Given a subgroup $F$ of $G$, let $E(F)$ denote the set of all $\varphi \in \operatorname{End}(G)$ such that $\varphi^{-1}(F) \subseteq F$.  Then $E(F)$ is a submonoid of the multiplicative monoid of the endomorphism ring $\operatorname{End}(G)$, and a subgroup $F$ of $G$ is an $E$-factroid of $G$ if and only if $E$ is a subset of $E(F)$.  Moreover, $E(F)$ is the set of all endomorphisms $\varphi$ of $G$ such that $\ker(G \stackrel{\varphi}{\to} G \to G/F) \subseteq \ker(G \to G/F)$, where $G \rightarrow G/F$ denotes the canonical projection.

Suppose, for instance, that $G = M$ is an $A$-module.  Then the canonical ring homomorphism $i: A \to \operatorname{End}_\Z(M)$ restricts to a monoid homomorphism $W(F) \to E(F)$, and $F$ is a $T$-factroid of $M$, where $T$ is a subset of $A$, if and only if $F$ is an $i(T)$-factroid of $M$, if and only if $i(T) \subseteq i(W(F))$.  In other words, one has $i(W(F)) = E(F)$.  Note also that $i(Z(A)) \subseteq \operatorname{End}_A(M)$, so, if $A$ is commutative, then any $T$-factroid of $M$ for $T \subseteq A$ is an $E$-factroid of $M$ for some $E \subseteq \End_A(M)$.  Thus in this case, one may avoid group endomorphisms of $M$ that are not $A$-linear.

The dual notion to that of an $E$-factroid of $G$ is the following: say that an \emph{$E$-submodule of $G$} is a subgroup $F$ of $G$ such that $\varphi(F) \subseteq F$ (or, equivalently, $F \subseteq \varphi^{-1}(F)$) for all $\varphi \in E$.  This extends the notion of a $T$-submodule of an $A$-module in the obvious way.  For any  subgroup $F$ of $G$, the set $S(F)$ of all $\varphi \in \operatorname{End}(G)$ such that $\varphi(F) \subseteq F$ is the largest subset (or subring) $E$ of $\operatorname{End}(G)$ such that $F$ is an $E$-submodule of $G$, and in fact $F$ is an $S(F)$-submodule of $G$ in the module-over-a-ring sense.    Moreover, $S(F)$ is equivalently the set of all $\varphi \in \operatorname{End}(G)$ such that $ \ker(G \to G/F) \subseteq \ker(G \stackrel{\varphi}{\to} G \to G/F)$, or, equivalently, such that $  \im(F \to G \stackrel{\varphi}{\to} G) \subseteq \im(F \to G)$, where $F \to G$ is the canonical inclusion.  Furthermore, one has $i(A(F)) = S(F)$ for any additive subgroup $F$ of an $A$-module $M$, where, again, $i$ is the canoncial ring homomomorphism $i: A \to \operatorname{End}_\Z(M)$.

The definitions and equivalences noted above situate the theory of $T$-factroids and $T$-submodules of an $A$-module developed in this paper in the broader context of the category of abelian groups, and, even more broadly, in the context of any abelian category.  A possible future direction is to undertake a study of the theory in these more general categorical contexts.
\end{rem}

\section{$W$-factroids of a ring}
\label{sec:rings}
Recall that by a \emph{$W$-factroid of $A$} we mean a $W$-factroid of the left  $A$-module ${}_AA$, that is, an additive subgroup $F$ of $A$ such that $a \in F$ whenever $w \in W$, $a \in A$, and $wa \in A$.   In this section, we investigate the $W$-factroids of $A$ for various examples and classes of rings $A$ and multiplicative subsets $W$.

\begin{example}\label{ex:integers}
We determine the $W$-factroids of $\Z$ for the various multiplicative subsets $W$ of $\Z$.

For an arbitrary nontrivial subgroup $n\Z$ of $\Z$, one has $$W(n\Z) = \{g \in \Z \mid (n\Z :_\Z g) \subseteq n\Z\} = \{g \in \Z: (g,n) = 1\},$$
    where also  $$W(0\Z) = \Z \setminus \{0 \}.$$
Thus, a  nontrivial subgroup  $n\Z$ of $\Z$ is a $W$-factroid of $\Z$ if and only if all of the elements of $W$ are relatively prime to $n$. 
In particular, if $W = \Z \setminus \{0\}$, then the only $W$-factroids of $\Z$ are $\{0\}$ and $\Z$.
Note that all of the multiplicative subsets $W(n\Z)$ of $\Z$ are saturated, as follows also from Proposition~\ref{pr:modulefactroid}.  Of course $A(n\Z) = \Z$ for all $n$, since there are no other subrings of $\Z$.
\end{example}

\begin{example}\label{ex:rationals}
We determine the $W$-factroids of $\Q$ for its various multiplicative subsets $W$.

    Let $A = M = \Q$.  One has $A(\Z) = \Z$, and therefore, by Proposition \ref{prop:Wunits}, $$W(\Z) = W(\Z) \cap \Q^\times = (A(\Z) \cap \Q^\times)^{-1} =  \{1/n: n \in \Z \setminus \{0 \}\}$$
    is largest submonoid $W$ of $\Q$ for which $\Z$ is a $W$-factroid of $\Q$.  Let $F$ be any proper subgroup of $\Q$.  Then $$A(F) = (F :_\Q F) = \Z[1/p: (1/p)F\subseteq F, \ p \text{ prime}]$$
    and therefore
   \begin{align*}
    W(F)& = (A(F) \cap \Q^\times)^{-1} \\ & = \langle p \text{ prime}: (1/p) F \subseteq F \rangle  \cdot \{1/n: n \in \Z\setminus \{0 \} \} \\
    & = \langle p \text{ prime}: (1/p) F \subseteq F \rangle \cdot W(\Z),
    \end{align*}
    where $\langle - \rangle$ denotes ``submonoid generated by,'' and where $\cdot$ denotes the monoid compositum of pointwise products.   In particular, if $P$ is any set of prime numbers and $F = \Z[1/p: p \in P]$, then $A(F) = F$ and $W(F) = \langle P \rangle  \cdot W(Z)$.  This, along with the characterization of the additive subgroups of $\Q$ \cite{BeZu-addQ}, effectively classifies all $W$-factroids of $\Q$.  
\end{example}

\begin{example}
    One can replace $\Z$ in Example \ref{ex:rationals} with any PID $A$ and $\Q$ with its quotient field $K$ and use the method employed in the example to characterize all of the  $A$-submodules of $K$ that are $W$-factroids of $K$ for the various multiplicative subsets $W$ of $K$.
\end{example}

\begin{defn}
    A \emph{factroid of $A$} is a $\reg(A)$-factroid of $A$, that is, it is an additive subgroup $F$ of $A$ such that $a \in F$ for any  $a \in A$ such that $ba \in F$ for some left nonzerodivisor $b$ of $A$.
\end{defn}

\begin{example}
        Since $W_A((0)) = \reg(A)$, a factroid of $A$ is equivalently a $W_A((0))$-factroid of $A$.  In particular, the ideals $(0)$ and $A$ are both factroids of $A$.  
\end{example}

\begin{example}\label{ex:ideals}
Suppose that $A$ is commutative.  By Proposition~\ref{pr:submoduleprimary}, any primary ideal of $A$ not intersecting $W$ is a $W$-factroid of $A$.
Thus, since any ideal that is maximal among the ideals not intersecting $W$ is prime, any such ideal is a $W$-factroid of $A$.

An immediate consequence of the above is that any primary ideal of $A$ containing only zerodivisors is a factroid of $A$. Since the minimal prime ideals of $A$ contain only zerodivisors, it follows that any minimal prime ideal of $A$ is a factroid of $A$.  In fact, since any minimal prime does not meet $A^\circ$, any minimal prime is an $A^\circ$-factroid of $A$.  More generally, the intersection of any collection of minimal primes of $A$ is a $A^\circ$-factroid of $A$.  It follows, then, that $\sqrt{0}$ is an $A^\circ$-factroid.
\end{example}

Example \ref{ex:ideals} generalizes as follows.

\begin{prop}\label{prop:intprimeary}
Suppose that $A$ is commutative. Any intersection $\bigcap  \q_i$ of primary ideals $\q_i$  of $A$ is an $(A \setminus \bigcup \sqrt{\q_i})$-factroid of $A$.  In particular, any intersection  $\bigcap  \p_i$ of prime ideals $\p_i$ of $A$ is an $(A\setminus \bigcup \p_i)$-factroid of $A$.
\end{prop}

\begin{proof}
    Let $I = \bigcap  \q_i$. If $w \in A\setminus \bigcup \sqrt{\q_i}$ and $r \in A$ with $wr \in I$, then, for all $i$, one has $w \notin \sqrt{\q_i}$, whence $r \in \q_i$, so that $r\in I$.  
\end{proof}

The following result essentially follows from  Proposition~\ref{pr:modulefactroid} and  \cite[Lemmas 3.1 and 8.1]{Kis-minprimes}.

\begin{prop}[{cf.\ \cite[Lemmas 3.1 and 8.1]{Kis-minprimes}}]\label{prop:Acircrad}
    Suppose that $A$ is commutative, and let $I$ be an ideal of $A$.
    \begin{enumerate}
        \item  $I$ is proper if and only if $W(I) \subseteq A\setminus I$, with equality if and only if $I$ is prime. 
        \item A prime ideal $\p$ of $A$ contains $I$ if and only if 
        $$W(\p) \subseteq \{a \in A \mid (\sqrt{I}:_A a) \subseteq \p\}$$
        with equality if and only if $\p$ is minimal over $I$.
        \item One has
    $$W(\sqrt{I}) =  \{a \in A \mid (\sqrt{I} :_A a) = \sqrt{I} \} = A\setminus \bigcup \operatorname{Min}(I),$$
    where $\operatorname{Min}(I)$ is the set of all prime ideals of $A$ that are minimal over $I$.
    \item     One has $$W(\sqrt{{0}}) = \{a \in A \mid (\sqrt{0} :_A a) = \sqrt{0}\} = A^\circ.$$ Equivalently, $A^\circ$ is the largest subset $T$ of $A$ such that $\sqrt{0}$ is a $T$-factroid of $A$.  
    \end{enumerate}
\end{prop}

\begin{proof}
  By Proposition~\ref{pr:modulefactroid}, $W(I)$ is the set of all nonzerodivisors of $A$ mod $I$.  It follows that $W(I) \subseteq A\setminus I$ if and only if $I$ is proper. But $I$ is prime if and only if $I$ is proper and every element of $A\setminus I$ is a nonzerodivisor mod $I$,  which in turn holds if and only if $A \setminus I \subseteq W(I) \subseteq A \setminus I$.  This proves statement (1).  Statement (2)  is follows from \cite[Lemmas 3.1 and 8.1]{Kis-minprimes} applied to the ring $A/I$.
  Finally, from statements (2) and (1)  and the fact that $\sqrt{I} = \bigcap_{\p \in \operatorname{Min}(I)} \p$ it follows that
  $$W(\sqrt{I}) = \bigcap_{\p \in \operatorname{Min}(I) } W(\p) = \bigcap_{\p \in \operatorname{Min}(I) }(A \setminus \p) = A\setminus \bigcup \operatorname{Min}(I),$$
  thus proving statement (3), from which (4) immediately follows.  Alternatively, statement (3) follows by lifting the equality $\reg(A/\sqrt{I}) =   (A/\sqrt{I})^\circ$ to $A$ (which holds since $A/\sqrt{I}$ is a reduced ring)  and  then applying the identity $W_A(\sqrt{I}) = \reg_{A/\sqrt{I}}(A)$ from Proposition \ref{pr:modulefactroid}.
\end{proof}

\begin{prop}
A nontrivial commutative ring $A$ is an integral domain if and only if $A$ is the only nonzero ideal of $A$ that is also a factroid of $A$.
\end{prop}

\begin{proof}
Suppose that $A$ is an integral domain.  Since the ring $A$ is nontrivial, one has $A \neq \{0\}$. Moreover, if $I$ is both a nonzero ideal and a factroid of $A$, then $r \in I$ for some nonzero (hence regular) $r$, which implies $ar \in I$, and therefore $a \in I$, for any $a \in A$, whence $I = A$.  Conversely, suppose that $A$ is the only nonzero ideal of $A$ that is also a factroid of $A$.  Since the ring $A$ is nontrivial, $A$ has some minimal prime ideal $P$.  Then $P$ is a proper ideal of $A$ that is also a factroid of $A$, and therefore $P = (0)$.  Thus $(0)$ is prime,  whence $A$ is an integral domain.
\end{proof}

\begin{prop}\label{prop:WcapM}
    One has  $[1]^{W \cap A^\times}_A = \Z\langle (W \cap A^\times)^{-1}\rangle$. Moreover, if $1 \in W$ and $F$ is a $W$-factroid of $A$, then the following are equivalent: \begin{enumerate}[(a)]
        \item $W \cap F \neq \emptyset$.
        \item $1 \in F$.
        \item $[1]^{W \cap A^\times}_A\subseteq F$.
        \item  $(W \cap A^\times)^{-1} \subseteq F$.
    \end{enumerate}
\end{prop}

\begin{proof}
By Proposition \ref{prop:unitfactroid}, 
$[1]^{W \cap A^\times}_A = \Z\langle(W \cap A^\times)^{-1}\rangle$.   Let $F$ be a $W$-factroid of $A$ containing an element $w$ of $W$.  Since $w1 \in F$, one has $1 \in F$.  Conversely, if $1 \in F$, then $F$ contains an element of $W$.  Finally, $1 \in F$ if and only if $[1]^{W \cap A^\times}_A\subseteq F$, since $F$ is a $(W \cap A^\times$)-factroid of $A$.  
\end{proof}

\begin{cor}\label{cor:WcapM}
    One has  $[1]^{A^\times}_A = \Z\langle A^\times \rangle$.  Moreover, the following are equivalent for a factroid $F$ of $A$: \begin{enumerate}[(a)]
        \item $F$ contains a left-regular element of $A$.
        \item $1\in F$.
        \item $A^\times \subseteq F$.
        \item $\Z\langle A^\times \rangle \subseteq F$.
        \item Any sum of finitely many units of $A$ lies in $F$.
    \end{enumerate}
\end{cor}

From the above, we obtain the following.  Note that the rings in question are called \emph{S-rings} in \cite{Ra-genu}.
\begin{thm}
Let $A$ be a ring for which every element is a sum of units (e.g., a local ring or any quotient or localization of $\Z$ at a multiplicative subset).  Then $A$ is the only factroid of $A$ that contains a left-regular element. Thus, if $A$ is moreover an integral domain, then the only factroids of $A$ are $\{0\}$ and $A$.
\end{thm}

Our next result implies that, for any ring $A$ of algebraic numbers, the only factroids of $A$ are $\{0\}$ and $A$.

\begin{prop}\label{pr:fieldext}
     Let $D$ be an integral domain with quotient field $K$, let $L/K$ be any algebraic field extension, and let $A$ be any subring of $L$ containing $D$.  Then for any $\reg(A)$-factroid $F$ of any $A$-submodule $M$ of $L$ such that $F$ contains $D$, $F$ also contains $A$.  In particular, $A$ is the only factroid of $A$ containing $D$.
\end{prop}

\begin{proof}
Let $M$ be an $A$-submodule of $L$, let $F$ be any $\reg(A)$-factroid of $M$ such that $D \subseteq F$, and let $0\neq \alpha \in A$.  Since $\alpha$ is algebraic over $K$, we can let $0 \neq h\in K[x]$ be the unique monic minimal polynomial of $\alpha$.  Choose  any nonzero $c \in D$ such that $g := ch \in D[x]$.  Since $g$ is irreducible in $K[x]$ and $\alpha \neq 0$, the polynomial $x$ does not divide $g$, so $0 \neq g(0) \in D \subseteq F$. But $g = -xf + g(0)$ for some polynomial $f\in D[x]$, so evaluating at $\alpha$ we get $\alpha f(\alpha) = g(0) \in F$.  Since $g(0) \neq 0$, we have $f(\alpha) \in A\setminus \{0\} = \reg(A)$, so $\alpha \in F$.  Since $\alpha \in A \setminus \{0\}$ was arbitrary, we have $A \subseteq F$.
The final statement of the proposition follows from the special case $M=A$.
\end{proof}

\begin{cor}
 If $A$ is any subring of the algebraic closure of $\Q$, then the only factroids of $A$ are $\{0\}$ and $A$.
\end{cor}

\begin{proof}
 If $F$ is any nonzero factroid of $A$, then there is some $0\neq u\in F$, so that $u \cdot 1 = u \in F$, whence $1\in F$, so $\Z \subseteq F$ since $\Z$ is the subgroup of $A$ generated by $1$.  The result then follows from Proposition~\ref{pr:fieldext}.
\end{proof}

The proof of the following proposition is postponed until Section \ref{sec:Euclidean}, as it is an immediate corollary of Theorem \ref{prop:normeuclidean},  which generalizes it to the Euclidean domains.

\begin{prop}\label{ex:polysimple}
Let $A=k[x]$, where $k$ is a field.  The  factroids of $A$ are precisely the sets \[
C_n := \{f\in A \mid \deg f \leq n\}
\]
for all $n \in \N \cup \{\pm \infty\}$.  Moreover, $A$ is not principal, but each $C_n$ for $n < \infty$ is principal, generated by any element of $A$ of degree $n$.
\end{prop}

\begin{example}\label{ex:polymore}
    In light of Examples~\ref{ex:integers} and \ref{ex:rationals}, it is natural to ask whether we can restrict our attention to saturated multiplicative sets. That is, is it always true that the $W$-factroids of $A$ and $\widetilde W$-factroids of $A$ coincide, where $\widetilde W$ is the saturation of $W$?  The answer is no.  For a counterexample, let $A = k[x]$, let $W$ be the set of nonzero polynomials of even degree, and let 
    $F = \spn_k\{1,x^2\}$.  Note that $\widetilde W = A^\circ$, since the square of any polynomial has even degree.  Hence by Proposition~\ref{ex:polysimple}, $F$ is not a $\widetilde W$-factroid of $A$.  However, we claim that $F$ is a $W$-factroid  of $A$.

    To see this, let $w\in W$ and $f\in A$ such that $wf\in F$.  Write $w = ax^2 + b$, with $a,b \in k$ and not both zero.  If $a=0$, then $w=b \in k^\times$ and since $F$ is a $k$-vector space and $bf \in F$, we have $f = b^{-1}bf \in F$.  Thus, we may assume $a\neq 0$, so that $\deg w = 2$.  But since $wf\in F$, we have $2+\deg f= \deg w + \deg f = \deg(wf) \leq 2$.  Thus, $\deg f=0$ so $f\in k \subseteq F$.
    
    Note that $W(F)$ cannot be saturated, since, if it were, then since $W \subseteq W(F)$ we would have $\widetilde W \subseteq W(F)$ and then $F$ would be a $\widetilde W$-factroid of $A$.  So this example shows that $W(F)$ need not be saturated.
    
    In fact, $W(F)$ is precisely the set of all nonzero polynomials in   $k[x]$ of degree not equal to 1.  To see the containment $\supseteq$, observe that any nonzero polynomial of degree $0$ or $2$ is in $F \subseteq W(F)$ while, if $\deg f \geq 3$ and $fg \in F$, then $g=0 \in F$ (since degrees are additive on nonzero polynomials), so that $f\in W(F)$.  For the opposite inclusion, let $f$ be a polynomial of degree $1$.  Write $f = ax+b$, where $0\neq a \in k$ and $b\in k$.  Then $(ax-b)f = a^2 x^2 - b^2 \in F$, but $ax-b \notin F$, so $ax+b \notin W(F)$.

    Note also that $A(F) = (F :_A F) = k$, since multiplying any positive-degree polynomial by $x^2 \in A$ would force it out of $F$.
\end{example}

\begin{example}\label{ex:kxy}
    Let $A = k[x_1 ,\ldots, x_n]$, where $k$ is a field and $n\geq 2$.  Let $f$ be an element of $A$ such $f +c$ is irreducible for all $c \in k$ (e.g., $f=x_1^2 + x_2^3$).  Then the principal factroid of $A$ generated by $f$ is the $k$-span of $\{1,f\}$. 

    In order to see this, it suffices to show that $U := \langle 1,f \rangle$ is a factroid of $A$, since it is obvious that any factroid containing $f$ must contain $U$ .  So let $g\in A^\circ$ and $h \in (U :_A g)$. Then $gh \in U$.  That is, there exist $a,b \in k$ with $gh = a +bf$.  If $b=0$, then $h \in k \subseteq U$, since any factor of a constant must be a constant.  If $b \neq 0$, then $\frac gb \cdot h= f + \frac{a}{b}$, which is irreducible by assumption, so either $h \in k$ (in which case $h\in U$) or $\frac gb \in k$.  In the latter case, it follows that $g\in k^\times$, so that $h = \frac bg \cdot f + \frac ag \in U$. Thus, $(U :_A g) \subseteq U$, so that since $g\in A^\circ$ was arbitrary, $U$ is a factroid of $A$.  Therefore, $U = [f]_A$.
\end{example}

Next, we determine the factroids of a direct product $A \times B$ of integral domains $A$ and $B$ in the case where $A \times B$ has more than one unit.

\begin{lemma}\label{lem:factproductgeneral}
Let $A$ and $B$ be rings, let $V \subseteq A$ and $W \subseteq B$ be multiplicative submonoids, let $M$ be an $A$-module and $N$ a $B$-module.  Let $F$ be a $V$-factroid of $M$ and $G$ a $W$-factroid of $N$.  Then $F \times G$ is a $(V \times W)$-factroid of $M \times N$.
\end{lemma}

\begin{proof}
Let $(v,w) \in V \times W$ and $(m,n) \in M \times N$ with $(v,w)(m,n) = (vm,wn) \in F \times G$.  Then $vm \in F$, so $m\in F$ since $F$ is a $V$-factroid of $M$, and $wn \in G$, so $n \in G$ since $G$ is a $W$-factroid of $N$.  Thus, $(m,n) \in F \times G$.
\end{proof}

\begin{thm}\label{thm:factproduct}
Let $A$ and $B$ be integral domains.  Assume that the ring $A \times B$ has more than one unit.  Then the factroids of $A \times B$ are precisely those subgroups of the form $F \times G$, where $F$ is a factroid of $A$ and $G$ a factroid of $B$.
\end{thm}

\begin{proof}
First, recall that $\reg(A \times B) = \reg(A) \times \reg(B) = A^\circ \times B^\circ$.
If $F$ is a factroid of $A$ and $G$ of $B$, it then follows from Lemma~\ref{lem:factproductgeneral} that $F \times G$ is a factroid of $A \times B$.

Conversely, let $H$ be a factroid of $A \times B$.  Let $\pi_1: A \times B \ra A$ and $\pi_2: A\times B \ra B$ be the projections. Set $F := \pi_1(H)$ and $G := \pi_2(H)$.  We first show that $F$ and $G$ are factroids of their ambient rings.  To see this, let $x\in A$ and $a\in A^\circ$ with $ax \in F$.  Then there is some $y\in G$ with $(ax,y) \in H$.  That is, $(a,1)(x,y) \in H$, so since $(a,1) \in (A\times B)^\circ$ and $H$ is a factroid, we have $(x,y) \in H$, whence $x\in F$.  Thus $F$ is a factroid of $A$, and by symmetry $G$ is a factroid of $B$.

Obviously $H \subseteq F \times G$. Now let $f\in F$.  Then there is some $y\in G$ with $(f,y) \in H$.  Since $(A \times B)^\times = A^\times \times B^\times$ is not a singleton, either $A^\times$ or $B^\times$ is not a singleton.  Without loss of generality assume the former, and let $1\neq \alpha \in A^\times$. Then $(\alpha^{-1},1)(\alpha f, y) = (f,y) \in H$, so since $(\alpha^{-1},1) \in \reg(A \times B)$, we have $(\alpha f, y) \in H$.  Subtracting, we have $(1-\alpha,1)(f,0) = ((1-\alpha)f, 0) = (f-\alpha f, y-y) = (f,y) - (\alpha f, y) \in H$, so since $(1-\alpha, 1) \in A^\circ \times B^\circ = (A \times B)^\circ$, we have $(f,0) \in H$.  We have shown that $F \times 0 \subseteq H$.

Now let $g \in G$.  Then there is some $x\in F$ with $(x,g) \in H$. But since $(x,0) \in H$, it follows that $(0,g) = (x,g) - (x,0) \in H$. 

Finally, take any $(f,g) \in F \times G$.  Since $(f,0), (0,g) \in H$, we have $(f,g) = (f,0) + (0,g) \in H$.  Thus, $H=F \times G$.
\end{proof}

\begin{example}
    The condition on the unit groups in Theorem~\ref{thm:factproduct} is necessary.  For instance, let $A = B = \mathbb F_2$.  Then $H = \mathbb F_2 \cdot (1,1) = \{(0,0), (1,1)\}$ is a factroid of $\mathbb F_2 \times \mathbb F_2$ but is not the product of a pair of factroids.
\end{example}

\begin{rem}
An approach such as that in Lemma~\ref{lem:factproductgeneral} and Theorem~\ref{thm:factproduct} will not characterize the $W$-factroids of a product of rings  for arbitrary $W$.  This is because not all multiplicative sets in $A \times B$ are Cartesian products.
\end{rem}

\begin{problem}
    Characterize the $W$-factroids of the rings $\Z \times \Z$ (resp., $\Q \times \Q$) for its various multiplicative subsets $W$.
\end{problem}

\begin{problem}
    Characterize the factroids of the domain $k[x,y]$, where $k$ is a field.
\end{problem}

\begin{problem}
    Let $K$ be a number field.  Characterize the $W$-factroids of $K$ (resp., $\mathcal{O}_K$) for its various multiplicative subsets $W$.
\end{problem}

\begin{problem}
 Characterize the $W$-factroids of the division ring $K = \Q(i,j)$ of all quaternions for its various multiplicative subsets $W$.
\end{problem}

\section{Generating $W$-factroids of a module}\label{sec:generation}

In this section, we describe various ways to produce $W$-factroids of a module.

The set of all $W$-factroids of $M$ is a partially ordered set under the relation $\subseteq$.  By Proposition~\ref{pr:intersect}, the poset of all $W$-factroids of $M$ has arbitrary infima and is therefore a complete poset.  The supremum of a set  of $W$-factroids is equal to the $W$-factroid generated by their union.  By our next proposition, sups of directed subsets are just unions.

 \begin{lemma}\label{lem:dirlim}
     Let $T \subseteq A$.  An arbitrary union of $T$-saturated subsets of $M$ is a $T$-saturated subset of $M$.  More generally, the $T$-saturation of the union of an arbitrary collection of $T$-subsets of $M$ is the union of their $T$-saturations in $M$.
 \end{lemma}

\begin{proof}
First, we claim that $(\bigcup_{\alpha \in \Lambda} S_\alpha :_M t) = \bigcup_{\alpha \in \Lambda} (S_\alpha :_M t)$ for any $t \in A$ and any collection $\{S_\alpha \mid \alpha \in \Lambda\}$ of subsets $S_\alpha$ of $M$. To prove this identity, note that the containment $\supseteq$ is clear.  Let $x \in (S:_M t)$, where $S = \bigcup_{\alpha \in \Lambda} S_\alpha$.  Then $tx  \in S$, whence $tx \in S_\beta$ for some $\beta \in \Lambda$, so that $x \in (S_\beta :_M t) \subseteq \bigcup_{\alpha \in \Lambda} (S_\alpha :_M t)$.  This proves the identity.  Finally, one has
\begin{align*}
    \operatorname{Sat}^T_M(S) &  = \bigcup_{w \in \langle T \rangle} (S:_M w) \\ & = \bigcup_{w \in \langle T \rangle}  \bigcup_{\alpha \in \Lambda} (S_\alpha :_M w) \\ & =  \bigcup_{\alpha \in \Lambda}  \bigcup_{w \in \langle T \rangle}  (S_\alpha :_M w) \\ & = \bigcup_{\alpha \in \Lambda}  \operatorname{Sat}^T_M(S_\alpha).
    \end{align*}
This completes the proof.
\end{proof}

\begin{prop}\label{pr:dirlim}
 Let $T\subseteq A$.  Any directed union of $T$-factroids of $M$ is a $T$-factroid of $M$.  More precisely, let $(\Lambda, \leq)$ be a directed set, and let $\{F_\alpha \mid \alpha \in \Lambda\}$ be a a collection of $T$-factroids of $M$ such that $F_\alpha \subseteq F_\beta$ whenever $\alpha \leq \beta$.  Then the union $\bigcup_{\alpha \in \Lambda} F_\alpha$ of the $F_\alpha$ is a $T$-factroid of $M$ and is therefore the smallest $T$-factroid of $M$ containing $F_\alpha$ for all $\alpha$.
\end{prop}

\begin{proof}
This follows from the lemma and the fact corresponding fact for $\Z$-submodules (i.e., for additive subgroups) of $M$.
\end{proof}

\begin{example}\label{ex:nosum}
The sum of two principal factroids of a commutative ring need not be a factroid.  Let $A = k[x,y]$, where $k$  is a field.  Let $F = \spn_k\{1,x\}$ and $G = \spn_k\{1,y^2-x\}$.  Then since $x+c$ and $y^2-x+c$ are irreducible for all $c \in k$, both $F$ and $G$ are principal factroids of $A$, by Example \ref{ex:kxy}.  Since $y$ is a factor of $y^2 = x+(y^2-x)$, any factroid of $A$ that contains $F$ and $G$ must contain $y$.  But $F+G = \spn_k\{1,x, y^2\}$ does not contain $y$.
\end{example}

Hereafter, per Proposition \ref{pr:Tfact}, we focus our attention on $W$-factroids of $M$, where $W$ is a multiplicative subset of $A$.

One has two simple operations of ``residuation,'' as described in Lemmas \ref{lem:factcolonmod} and \ref{lem:factcolonring} below.

\begin{lemma}\label{lem:factcolonmod}
For any $W$-factroid $F$ of $M$ and any subset $T \subseteq Z(A)$, $(F :_M T)$ is a $W$-factroid of $M$.
\end{lemma}

\begin{proof}
To see that $(F:_M T)$ is a subgroup of $M$, first note it contains $0$. Next let $x,y \in (F:_MT)$.  Then for any $t\in T$, we have $tx, ty \in F$.   Since $F$ is a subgroup of $M$, $t(x-y) = tx-ty \in F$, so since $t\in T$ was arbitrary, $x-y \in (F :_MT)$.

Now let $x \in M$ and $w\in W$ such that $wx \in (F:_MT)$.  Then for any $t\in T$, we have $wtx = twx \in F$.  Since $w \in W$ and $F$ is a $W$-factroid, $tx \in F$.  Since $t \in T$ was arbitrary, $x \in (F:_MT)$.
\end{proof}

\begin{rem}
 Lemma~\ref{lem:factcolonmod} says, equivalently, that $W((F:_M T))  \supseteq W(F)$ for any multiplicative set $W$, any $W$-factroid of $M$, and any subset $T \subseteq Z(A)$. In fact the containment can be proper.  Indeed, let $A=M=\Z$, $F = 6\Z$, and $T = \{2\}$, so that $(F :_M T) = 3\Z$.  Then $2 \in W((F :_M T)) \setminus W(F)$.
\end{rem}

\begin{lemma}\label{lem:factcolonring}
    For any $W$-factroid $F$ of $M$ and any subset $S \subseteq M$, $(F :_AS)$ 
    is a $W$-factroid of $A$.
\end{lemma}

\begin{proof}
For the subgroup property, first note that it contains $0$. Next let $a,b \in (F :_A S)$.  Then for any $x\in S$, we have $ax,bx \in F$, so that since $F$ is a group, $(a-b)x = ax-bx \in F$.  Since $x\in S$ was arbitrary, $a-b \in (F:_AS)$

Now let $a\in A$ and $w\in W$ such that $wa \in (F:_AS)$.  Then for any $x\in S$, we have $wax \in F$.  Thus, $ax \in (F :_M w) \subseteq F$ since $F$ is a $W$-factroid.  Since $x\in S$ was arbitrary, $a\in (F :_AS)$.  Therefore $(F:_A S)$ is a $W$-factroid of $A$.
\end{proof}

In the next few results, we show that generation of $W$-factroids interacts well with multiplication of subsets of $A$ and subsets of $M$, especially when the subset of $A$ is in the center of $A$.  This can be recast (though we shall not do so here) in the language of \emph{modules over a quantale}, a notion important in lattice theory (see e.g. \cite{Ros-quanbang}).

\begin{prop}\label{prop:multsubsets}
Let $S \subseteq M$ and $T \subseteq A$.  One has $$[T]^W_A S \subseteq [TS]^W_M.$$
Moreover, if $T \subseteq Z(A)$, then 
$$[T]^W_A [S]^W_M \subseteq [TS]^W_M.$$
\end{prop}

\begin{proof}
    Clearly one has $T \subseteq ([TS]^W_M:_A S)$, where the colon over $A$ is a $W$-factroid of $A$ by Lemma \ref{lem:factcolonring}. It follows that $[T]^W_A \subseteq ([TS]^W_M:_A S)$, which is equivalent to the first statement.  Suppose that $T \subseteq Z(A)$.  Since $S \subseteq ([TS]^W_M:_M T)$ and the colon over $M$ is a $W$-factroid of $M$ by Lemma \ref{lem:factcolonmod}, one has $[S]^W_M \subseteq ([TS]^W_M:_A T)$, and therefore $T [S]^W_M \subseteq [TS]^W_M$.  It follows, then, that 
    $$[T]^W_A [S]^W_M \subseteq  [T[S]^W_M ]^W_M \subseteq  [[TS]^M_M]^W_M = [TS]^W_M.$$
    This proves the second statement.
\end{proof}

\begin{cor}\label{cor:WAWMmul}
Let $x \in M$ and $T \subseteq A$.  Then $[T]^W_A x \subseteq [Tx]^W_M$.
\end{cor}

\begin{cor}\label{cor:newa}
Let $S \subseteq M$ and $a \in Z(A)$.  Then $$a[S]^W_M \subseteq [a]^W_A[S]^W_M \subseteq [aS]^W_M.$$ Equalities hold if $a$ is a unit of $A$.
\end{cor}

\begin{proof} 
The first claim is clear, and the second follows by applying the first  to $aS \subseteq M$ and $a^{-1} \in Z(A)$.
\end{proof}

\begin{example}\label{ex:noncentral}
The condition that $a\in Z(A)$ is necessary in Corollary~\ref{cor:newa}.  To see this, consider the ring $A = k\langle x,y \rangle$ of polynomials in noncommuting variables $x,y$ over a field $k$.  Set $S := k[x]y$, and $W := \{1,x,x^2,x^3, \ldots\} =$ the multiplicative submonoid of $A$ generated by $x$.  If $F$ is a $W$-factroid containing $S$, then it must contain $k[x] \oplus k[x]y$.  But the latter set is already a $W$-factroid of $A$.  Hence, $[S]^W_A = k[x] \oplus k[x]y$.  On the other hand, $yS = yk[x]y$ is already a $W$-factroid of $A$, since it is an additive subgroup such that no element of it has a left $W$-factor other than $1$.  Thus, $[yS]^W_A = yS= yk[x]y$.  Hence, $yx \in y[S]^W_A \setminus [yS]^W_A$.
\end{example}

\begin{cor}
 Let $V$ be a multiplicative submonoid of $Z(A)$.  Then $[V]^W_A$ is a $\Z\langle V \rangle$-subalgebra of $A$ and is the smallest subring of $A$ containing $V$ that is a $W$-factroid of $A$.  Let $F$ be any $W$-factroid of $M$.  Then the following conditions are equivalent.
 \begin{enumerate}[(a)]
        \item $VF \subseteq F$.
     \item $F$ is a $V$-submodule of $M$.
     \item $F$ is a $\Z\langle V \rangle$-submodule of $M$.
       \item $F$ is $[V]^W_A$-submodule of $M$.
     \item $V \subseteq A(F)$.
     \item $[V]^W_A \subseteq A(F)$.
     \item $\Z\langle V \rangle \subseteq A(F)$.
     \item $A(F)$ is a $[V]^W_A$-subalgebra of $A$.  
     \item  $A(F)$ is a $\Z\langle V \rangle$-subalgebra of $A$.  
 \end{enumerate}
\end{cor}

\begin{proof}
    By Proposition \ref{prop:multsubsets}, one has $1 \in [V]^W_A$ and $[V]^W_A [V]^W_A \subseteq [VV]^W_A \subseteq [V]^W_A$, and therefore $[V]^W_A$ is a subring of $A$.  Moreover, since $V[V]^W_A \subseteq [V]^W_A$, the ring $[V]^W_A$ is also a $V$-submodule, hence a $\Z\langle V \rangle$-subalgebra, of $A$.  The equivalence of the nine conditions is easy to check.
\end{proof}

\begin{cor}\label{cor:bracket1}
      The smallest $W$-factroid $[1]^W_A$ of $A$ containing $1$ is the smallest subring of $A$ that is a $W$-factroid of $A$.  Let $F$ be any $W$-factroid of $M$.  Then $F$ is a $[1]^W_A$-submodule of $M$, $[1]^W_A$ is a subring of $A(F)$, and $A(F)$ is a $[1]^W_A$-subalgebra of $A$. 
\end{cor}

Next, we provide a more explicit construction of the $T$-factroid $[S]^T_M$ of $M$ generated by a subset $S$ of $M$.

\begin{defn}
\label{cons:indfact}
Let $T \subseteq A$ and $S \subseteq M$.  We inductively define additive subgroups $F_i^T(S)$ of $M$ for all positive integers $i$, as follows. First, let $F_1^T(S)$ denote the additive subgroup of $M$ generated by the $T$-saturation $\operatorname{Sat}^T_M(S)$ of $S$ in $M$.   For $i\geq 2$, we inductively define $F_i^T(S) := F_1^T
(F_{i-1}^T(S))$.  In other words, $F_i^T(S)$ is the $i$-fold application of the operation of $F_1^T$ to $S$.  Since $S \subseteq F_1^T(S)$, one has $S \subseteq F_1^T(S) \subseteq F_2^T(S) \subseteq \cdots$.  Moreover, it is clear that $S$ is a $T$-factroid of $M$ if and only if $F_1^T(S) \subseteq S$, which is in turn equivalent to the inclusions  $F_i^T(S) \subseteq S$ holding for all $i \geq 1$.
\end{defn}

\begin{prop}\label{pr:factroidfrombelow}
Let $T \subseteq A$ and $S \subseteq M$. Then  $[S]_M^T = \bigcup_{i\geq 1} F_i^T(S)$, where $F_i^T(S)$ is as in Definition~\ref{cons:indfact}.
\end{prop}

\begin{proof}
Set $F(S) := \bigcup_{i\geq 1} F_i^T(S)$.  We need to show that $F(S)$ is a $T$-factroid of $M$, and that $F(S)$ is contained in any $T$-factroid that contains $S$.

For the latter, let $G$ be a $T$-factroid of $M$ that contains $S$.  Then $F_1^T(S) \subseteq F_1^T(G) \subseteq G$, so that $F_i^T(S) \subseteq G$ for all $i$, whence $F(S) \subseteq G$. 

For the former, first note that $F(S)$ is a directed union of additive subgroups (or $\{1\}$-factroids) of $M$ and is therefore an additive subgroup of $M$.  Moreover, by Proposition \ref{pr:dirlim}, $\operatorname{Sat}^T_M(F(S)) = \bigcup_{i\geq 1} \operatorname{Sat}^T_M(F_i^T(S)) \subseteq \bigcup_{i\geq 1} F_{i+1}^T(S) = F(S)$, so that $F(S)$ is $T$-saturated. Therefore, $F(S)$ is a $T$-factroid of $M$.

Here is an alternative and more explicit proof.  Let $W : = \langle T \rangle$, and let $x,y \in F(S)$.  Then since the union is nested, there is some $i$ with $x,y \in F_i^T(S)$.  Thus, $x-y \in F_i^T(S) \subseteq F(S)$. Next, let $z \in M$ and $w \in W$ such that $wz \in F(S)$.  Then there is some $i$ with $wz \in F_i^T(S)$, so that $z \in F_1^T(F_i^T(S)) = F_{i+1}^T(S) \subseteq F(S)$, completing the proof that $F(S)$ is a $W$-factroid.  Hence $[S]_M^T \subseteq F(S)$.
\end{proof}

The corollary below provides an alternative proof of Corollary \ref{cor:newa}.

\begin{cor}\label{cor:hinside}
Let $T \subseteq A$ and $S \subseteq M$. Let $F_i^T(-)$ be as in Definition~\ref{cons:indfact}. \begin{enumerate}
    \item For any $i\geq 0$, one has $F_i^T(T) S \subseteq F_i^T(TS)$.  Hence, $[T]^T_A S \subseteq [TS]^W_M$.
    \item Suppose $T \subseteq Z(A)$.  Then for any $i\geq 0$, $T F_i^T(S) \subseteq F_i^T(TS)$.  Hence $T[S]^T_M \subseteq [TS]^T_M$.
\end{enumerate}
\end{cor}

\begin{proof}
In both cases, the second statement follows from the first, by Proposition~\ref{pr:factroidfrombelow}.  So it suffices to prove the first statement of each of (1), (2).

(1): Let $r \in F_1^W(T)$ and $s\in S$.  Then there exist $a_j \in A$, $w_j \in W$ with $r = \sum_{j=1}^n a_j$ and $\pm w_j a_j \in T$.  But then $rs = \sum_{j=1}^n a_j s$, and $\pm w_j (a_j s) \in TS$.  Hence, $rs \in F_1^W(TS)$, and hence $F_1^W(T)S \subseteq F_1^W(TS)$.  Replacing $T$ by $F_i^W(T)$ inductively, one obtains the first statement.

(2): Since $T \subseteq Z(A)$, it is clear that, for any $b\in A$ and $t\in T$, we have $t(S :_M b) \subseteq (TS :_M b)$.  So let $z\in F_1^W(S)$.  Then there exist $z_j \in M$ and $w_j \in W$ with $z =\sum z_j$ and $z_j \in \pm(S:_M w_j)$.  Hence, for any $t\in T$, we have $tz_j \in \pm (TS :_M w_j)$, and $tz = \sum_j tz_j$, so $tz \in F_1^W(TS)$.  Replacing $S$ by $F_i^W(S)$ inductively, one obtains the first statement. 
\end{proof}

\begin{example}
    There are many examples where $[S]^W_M \supsetneq F_1^W(S)$. For instance, let $A = M = k[x]$, $W = A\setminus \{0\}$, and $S = \{f\}$, where $k$ is a field and $f$ is a monic irreducible polynomial in $k[x]$.  The monic factors of $f$ are $1$ and $f$, so that $F_1^W(f) = \spn_k\{1,f\}$.  But by Proposition~\ref{ex:polysimple} one has $[f]^W = \spn_k\{1,x,x^2, \ldots, x^n\}$, where $n = \deg f$.  Thus, one has $F_1^W(f) = [f]^W_A$ if and only if $\deg f \leq 1$.
\end{example}

\section{$T$-regular $W$-factroids of a module}\label{sec:regular}

In this section, we consider another general construction of factroids, which generally gives a larger set than $[-]_M^W$.  Here, $T$ will be an arbitrary subset of $A$, and $W$ a multiplicative subset of $A$.

\begin{defn}\label{defn:regular}
Let $T$ be a subset of $A$.  A $W$-factroid $F$ of $M$ is \emph{$T$-regular} if $([hF]_M^W :_M h) = F$ for all $h\in T$, or, equivalently, if $x\in F$ for any $x\in M$ such that 
$hx \in [hF]_M^W$ for some $h\in T$.
\end{defn}

\begin{rem}\label{rem:Tregbasics}
 If $T' \subseteq T$, then every $T$-regular $W$-factroid is $T'$-regular.  If $T \subseteq Z(A)^\times = A^\times \cap Z(A)$, 
then any $W$-factroid is $T$-regular, 
by Corollary~\ref{cor:newa}.

Observe also any $W$-factroid $F$ of $M$ is $(W \cap A(F))$-regular.  Indeed, let $h \in W \cap A(F)$ and $x \in M$ with $hx \in [hF]^W_M$.  Since $h \in W$ and $[hF]^W_M$ is a $W$-factroid, one has $x \in [hF]^W_M$.  Then, since $h \in A(F) = 
(F :_A F)$, one has $hF \subseteq F$ and therefore $x \in [F]^W_M = F$.  

Since an additive subgroup $F$ of $M$ is a $W$-submodule of $M$ if and only if $W \cap A(F) = W$, a $W$-factroid of $M$ that is also a $W$-submodule of $M$ is automatically $W$-regular.
\end{rem}

\begin{rem}\label{rem:TofF}
Let $F$ be a $W$-factroid of $M$.  Let $T^W_M(F)$ be the set of all $h \in A$ such that $([hF]_M^W :_M h) = F$.  Then $F$ is $T$-regular if and only if $T \subseteq T^W_M(F)$; that is, $T^W_M(F)$ is the largest subset $T$ of $A$ such that the $W$-factroid $F$ is $T$-regular.  Note that $$ T^W_M(F) = \bigcup_{a \in Z(A)}(T^W_M(F):_A a) \supseteq Z(A)^\times.$$  In other words, $T^W_A(F)$ is $Z(A)$-saturated; equivalently,  $F$ is $T$-regular if and only if $F$ is $\operatorname{Sat}^{Z(A)}_M(T)$-regular.  Moreover, by the previous remark one has $$(W \cap A^\times)^{-1} \subseteq W \cap A(F) \subseteq T_M^W(F).$$  In other words, every $W$-factroid is $(W \cap A(F))$-regular, and therefore also $(W\cap A^\times)^{-1}$-regular.
\end{rem}

\begin{problem}
    Must one have $A^\times \subseteq T^W_M(F)$?  Equivalently, must any $W$-factroid be $A^\times$-regular?
\end{problem}

\begin{problem}
    Must $T^W_M(F)$ be multiplicatively closed?  Equivalently, is a $W$-factroid $T$-regular if and only if it is $\langle T \rangle$-regular, where  $\langle T \rangle$ is the submonoid of $A$ generated by $T$?  
\end{problem}

\begin{lemma}\label{lem:capregfactroids}
Let $T \subseteq A$.  Any intersection of $T$-regular $W$-factroids of an $A$-module $M$ is $T$-regular.
\end{lemma}

\begin{proof}
Let $\{F_\alpha\}_{\alpha \in \Lambda}$ be a set of $T$-regular $W$-factroids in $M$, and let $F = \bigcap_\alpha F_\alpha$.  Let $h\in T$ and $x \in M$ such that $hx \in [hF]_M^W$.  
Since $[-]^W_M$ is order-preserving and $hF \subseteq hF_\alpha$ for all $\alpha$, it follows that $hx \in [hF_\alpha]^W_M$, so that $x\in ([hF_\alpha]^W_M :_M h) = F_\alpha$
Thus, $x \in \bigcap_\alpha F_\alpha = F$.
\end{proof}

Consequently, we can make the following definition.

\begin{defn}\label{def:Treggen}
  Let $T \subseteq A$.  For any subset $S$ of $M$, the \emph{$T$-regular $W$-factroid of $M$ generated by $S$} is the smallest $T$-regular $W$-factroid \[
    [S]_M^{W,T} := \bigcap \{F \mid F \text{ is a $T$-regular $W$-factroid of $M$ and $S \subseteq F$}\}.
    \]
    of $M$ containing $S$.  When $M$ is understood, we abbreviate $[S]_M^{W,T}$ to $[S]^{W,T}$.   When $S = \{x\}$ for some $x\in M$, we write $[x]_M^{W,T}$ or $[x]^{W,T}$ in place of $[\{x\}]^{W,T}_M$.
\end{defn}

\begin{rem}
      Let $S \subseteq M$.  If $T \subseteq A^\times$, then $[S]^{W,T}_M = [S]^W_M$. Moreover, if $T' \subseteq T \subseteq A$, then $[S]^{W,T'}_M \subseteq [S]^{W,T}_M$.  Consequently, one has $[S]^{W,T}_M \supseteq [S]^W_M$ for all $T \subseteq A$.
      \end{rem}

\begin{rem}
      If $A$ is commutative, then $[\emptyset]^{W,W}_M = [\emptyset]_M^W = \ker(M \rightarrow W^{-1}M)$ is the smallest ($W$-regular) $W$-factroid of $M$.
\end{rem}

In Definition~\ref{cons:indfact} and Proposition~\ref{pr:factroidfrombelow}, we showed a way to construct $[S]^W_M$ from \emph{below} -- i.e., as a union of certain subsets of $M$ rather than an intersection.  The following construction of $G^{W,T}_M(S)$ is an attempt at doing a similar thing for $T$-regular $W$-factroids.  However, as we shall see, this is only fully successful when $T$ is a multiplicative submonoid of the center of $A$.  Thus, in the commutative case the story is satisfying, but in the noncommutative case, questions remain.

\begin{defn}
Let $T \subseteq A$ and $S \subseteq M$.  Set 
\[G_M^{W,T}(S) = \bigcup_{h \in T} ([hS]_M^W :_M h)   = \{x \in M: hx \in [hS]^W_M \text{ for some } h \in T\}.
\]
     When $T = W$, we write $G^W_M(S)$ for $G^{W,T}_M(S)$.  When $M$ is understood, we write $G^{W,T}(S)$ for $G_M^{W,T}(S)$ and  $G^W(S)$ for $G^W_M(S)$.  We also set the obvious conventions when $S$ is a singleton.
\end{defn}

\begin{lemma}\label{lem:Gsubset}
Let $T \subseteq A$ and $S \subseteq M$.  Then $G_M^{W,T}(S) \subseteq [S]_M^{W,T}$. 
\end{lemma}

\begin{proof}
Let $G$ be any $T$-regular $W$-factroid containing $S$. Let $x\in G_M^{W,T}(S)$; then for some $h\in T$, $hx \in [hS]_M^W \subseteq [hG]_M^W$, so that $x \in ([hG]^W_M :_M h) = G$ since $G$ is $T$-regular.  The lemma follows.
\end{proof}

\begin{lemma}\label{lem:Gmulinside}
    Let $T \subseteq A$, $S \subseteq M$, and $c \in Z(A)$.  Then $cG^{W,T}(S) \subseteq G^{W,T}(cS)$.  Equality holds if $c$ is a unit of $A$.
\end{lemma}

\begin{proof}
Let $x \in G^{W,T}(S)$.  Then for some $h\in T$, we have $hx \in [hS]_M^W$.  Then $hcx = chx  \in c  [hS]_M^W \subseteq [chS]_M^W = [hcS]_M^W$ by Corollary~\ref{cor:newa}.  Thus, $cx \in ([hcS]_M^W :_M h) \subseteq G^{W,T}(cS)$.  Applying the given inclusion to $c^{-1} \in Z(A)$ and $cS \subseteq M$, one sees that the reverse inclusion holds if $c$ is a unit of $A$.
\end{proof}

\begin{prop}\label{pr:G}
Let $W$ be a multiplicative subset of $A$, let $V$ be a multiplicative submonoid of $Z(A)$, and let $S$ be a subset of $M$. Then $G^{W,V}_M(S)$ is the smallest $V$-regular $W$-factroid of $M$ containing $S$.   Equivalently, $[S]^{W,V}_M = G^{W,V}_M(S)$.
\end{prop}

\begin{proof}
First we prove $G=G^{W,V}(S)$ is a $W$-factroid of $M$.  

Let $z,z' \in G$.  Then there exist $h,h' \in V$ with $hz \in [hS]_M^W$ and $h'z' \in [h'S]_M^W$.  It then follows from Corollary~\ref{cor:newa} that $hh'z, hh'z' \in [hh'S]_M^W$.  Thus, $hh'(z-z') = hh'z - hh'z' \in [hh'S]_M^W$, where $h h' \in V$, whence $z-z' \in G$.  Since $G$ is nonempty, it follows that it is an additive subgroup of $M$.

Now let $x \in M$ and $r\in W$ such that $rx \in G$.  Then there is some  $h\in V$ with $rhx = hrx \in [hS]_M^W$.  Thus, $hx \in ([hS]_M^W :_M r)  \subseteq [hS]_M^W$ (since $[hS]_M^W$ is a $W$-factroid), whence $x \in G$.  Thus, $G$ is a $W$-factroid of $M$.

Next, to see that $G$ is $V$-regular, let $h\in V$ and $z\in ([hG]_M^W :_M h)$.  Then $hz \in [hG]_M^W$.  But $hG=hG^{W,V}(S) \subseteq G^{W,V}(hS)$ by Lemma~\ref{lem:Gmulinside}, so since $[-]_M^W$ is order-preserving, we have $hz \in [G^{W,V}(hS)]_M^W = G^{W,V}(hS)$, with the last equality because $G^{W,V}(hS)$ is a $W$-factroid.  Thus for some $h'\in V$, we have $h'hz \in [h'hS]_M^W$, whence $z \in ([h'hS]_M^W :_M h'h) \subseteq G$ since $h'h\in V$.

Finally, since $G$ is a $V$-regular $W$-factroid of $M$ containing $S$, and since  $[S]^{W,V}_M \supseteq G$ by Lemma \ref{lem:Gsubset},  equality must hold.
\end{proof}

\begin{example}\label{ex:Gnotaddcl}
Equality need not hold in Proposition~\ref{pr:G} when $V \nsubseteq Z(A)$.  Indeed, $G^{W,V}_M(S)$ need not even be additively closed.  For an example of this, let $A$ and $M$ be as in Example~\ref{ex:0notaddcl}, let $W = \{1\}$ (so that $[-]^W_M$ is just additive subgroup generation), let $S = \{0\}$, and let $V = \langle x,y \rangle$.  Then $(1,0) \in (0 :_M x) = ([0]^W_M :_M x)$, and $(0,1) \in (0 :_M y) = ([0]^W_M :_M y)$, so both are in $G^{W,V}_M(0)$, but for any $g\in V$, $(\bar g, \bar g) = g \cdot (1,1) \notin [g0]^W_M = 0$, so $(1,1)= (1,0) + (0,1) \notin G^{W,V}_M(0)$.
\end{example}

\begin{problem}
     Find a formula for $[S]^{W,T}_M$ analogous to the formula $[S]_M^{W} = \bigcup_{n = 1}^\infty F_i^W(S)$ that works for arbitrary subsets $T$ of $A$.
\end{problem}

\begin{example}
All of the factroids $C_n$ of $A = k[x]$, as classified in Proposition~\ref{ex:polysimple}, are $\reg(A)$-regular, since the degree function is additive.  Equivalently, one has $[S]_A^{\reg(A)} = G_A^{\reg(A)}(S)$ for all subsets $S$ of $A$.
\end{example}

\begin{rem}\label{rem:Gnottriv}
It is reasonable to ask: is it always true that $G^W_A(S) = [S]_A^W$?  The answer is no, even when $A$ is a domain, $W = \reg(A)$, and $S$ is a singleton.  See Example \ref{ex:Gnotbracket} for a specific example.
\end{rem}

\section{$W$-Egyptian fractions}\label{sec:Egyptian}

In this section, we provide some applications of these factroid constructions to $W$-Egyptian fractions and $W$-reciprocal complements, defined as follows.

\begin{defn}\label{def:recip}
Let $A$ be a commutative ring and $W$ a multiplicative set.  Then the set of all sums $\{\frac 1{w_1} + \cdots + \frac 1{w_n} \in W^{-1}A \mid n \in \N,\ w_i \in W \}$ is the nonunital (i.e., not necessarily unital) subring $R_W(A)$ of $A$ generated by the subset $W^{-1}  = \{\frac 1w \mid w \in W\}$ of $W^{-1}A$, which we call the $W$-\emph{reciprocal complement} of $A$.  If $1 \in W$, then $R_W(A)$ is a unital subring of $A$.

An element of $W^{-1}A$ is {\it $W$-Egyptian} in the sense of \cite{nme-Emult} if it lies in $R_W(A)$.  An element $a$ of $A$ is {\it $W$-Egyptian} if $a/1$ is $W$-Egyptian.  A ring $A$ is said to be {\it $W$-Egyptian}  if every element of $A$ (or, equivalently, of $W^{-1}A$) is $W$-Egyptian.  This condition holds if and only if $R_W(A) = W^{-1}A$, which in turn is equivalent to the condition that the image of the localization map $A \ra W^{-1}A$ is contained in $R_W(A)$. 

More generally, let $M$ be an $A$-module and $S \subseteq M$ a subset.  we let  $R_W^S(M)$ be all sums in $W^{-1}M$ of elements of the form $\frac sw$, where $s\in S$ and $w\in W$.  Note that if $S'$ is the image of $S$ in the localization map $M \ra W^{-1}M$, then $R_W^S(M)$ is the $R_W(A)$-submodule of $W^{-1}M$ generated by $S'$.
\end{defn}

\begin{rem}
Let $A$, $M$, $S$, $W$ be as above.  As per the notation in Corollary~\ref{cor:duality}, let $\frac S1 := \{s/1 \in W^{-1}M \mid s\in S\}$, i.e., the image of the set $S$ in the localization map $M \ra W^{-1}M$.  Then $R^S_W(M) = R_W(A)\frac S1$.  That is, it is the $R_W(A)$-submodule of $W^{-1}M$ generated by $\frac S1$.

Note also that if $M=A$ and $S = \{1\}$, then $R_W^S(M) = R_W(A)$.
\end{rem}

\begin{example}
Let $A$ be an integral domain with fraction field $K$. The subring $R(A) := R_{\reg(A)}(A)$ of $K$ is the \emph{reciprocal complement of $A$} (in the sense of \cite[Definition 2.1]{nme-Euclidean}), an element of  $K$ is said to \emph{$A$-Egyptian} (or \emph{Egyptian} if the context is clear) if it is $\reg(A)$-Egyptian, and the domain $A$ is said to be \emph{Egyptian} (in the sense of \cite{GLO-Egypt}) if $A$ is $\reg(A)$-Egyptian. 
\end{example}

\begin{lemma}\label{lem:inthelp}
 Let $A$ be a commutative ring, $W$ a multiplicative submonoid of $A$, $M$ an $A$-module, and $S \subseteq M$. Let $x\in M$ and $b\in W$ with $x \in [bS]^W_M$.  Then $x/b \in R_W^S(M)$.
\end{lemma}

\begin{proof}
We show inductively on $i$ that if $x \in F_i^W(bS)$, then $x/b \in R_W^S(M)$.  We may assume $x\neq 0$.

Case $i=0$: Then $x=bs$ for some $s\in S$, so $x/b = s/1 \in R_W^S(M)$. 

Case $i>0$, given true for smaller $i$: Then there are nonzero $x_j\in M$, $c_j \in W$ such that  $x = \sum_{j=1}^n \pm x_j$ and $c_j x_j \in F_{i-1}^W(bS)$.  By inductive hypothesis, we have $\frac{c_j x_j}b \in R_W^S(M)$.  Write $
\frac {c_j x_j}b = \sum_{h=1}^{t_j} \frac {s_{hj}}{u_{hj}},
$
where $t_j \in \N$, $s_{hj} \in S$, and $u_{hj} \in W$.  Then \[
\frac xb = \sum_{j=1}^n \pm \frac{x_j}b = \sum_{j=1}^n \frac {\pm 1}{c_j} \frac{c_j x_j}b = \sum_{j=1}^n \frac {\pm 1}{c_j} \sum_{h=1}^{t_j} \frac {s_{hj}}{u_{hj}} = \sum_{j=1}^n \sum_{h=1}^{t_j} \frac {\pm s_{hj}}{c_j u_{hj}} \in R_W^S(M).\qedhere \]
\end{proof}

We then have the following result characterizing elements of $R_W^S(M)$.

\begin{thm}\label{thm:recipG}
Let $A$ be a commutative ring,  let $W$ be a multiplicative submonoid of $A$, let $M$ be an $A$-module, let $S \subseteq M$ a subset, and let $\alpha \in W^{-1}M$.  The following are equivalent: 
\begin{enumerate}[(a)]
    \item $\alpha \in R_W^S(M)$.
    \item There exist $x\in M$ and $b \in W$ such that $\alpha=x/b$ and $x \in [bS]^W_M$.
   \item There exist $x\in M$ and $b \in W$ such that $\alpha=x/b$ and $x \in G^W_M(bS)$.
    \item For any $x\in M$ and $b\in W$ such that $\alpha=x/b$, we have $x\in G^W_M(bS)$.
    \item For any $x\in M$ and $b\in W$ such that $\alpha=x/b$, there is some $h\in W$ with $hx \in F_1^W(hbS)$.
\end{enumerate}
 Consequently, in the special case where $W$ consists of $M$-regular elements of $A$, so that $M$ can be identified with its image in $L :=W^{-1}M$, we have \[
R_W^S(M) = \bigcup_{b\in W} \left([bS]^W_M :_L b\right).
\]
\end{thm}

\begin{proof}
(e) $\Rightarrow$ (d): This holds because $F_1^W(hbS) \subseteq [hbS]^W_M$.

(d) $\Rightarrow$ (c): Clear.

(c) $\Rightarrow$ (b): Since $x\in G^{W}_M(bS)$, there is some  $h\in W$ with $hx \in [hbS]^W_M$, but $hx/hb=x/b=\alpha$, so (b) follows.

(b) $\Rightarrow$ (a): This follows from Lemma~\ref{lem:inthelp}.

(a) $\Rightarrow$ (e): Let $x\in M$ and $b\in W$ such that $\alpha=x/b$.  There exist $r_1, \ldots, r_m \in W$ and $s_1, \ldots, s_m \in S$ such that \[
\frac xb = \sum_{i=1}^m \frac {s_i}{r_i} = \frac{\sum_{i=1}^m \widehat{r_i}s_i}r,
\]
where $\widehat{r_i} = \prod_{j\neq i} r_j$ and $r = \prod_{i=1}^m r_i = r_j\widehat{r_j}$ for all $j$.  Then there is some $w\in W$ such that $wrx = wb(\sum_{i=1}^m \widehat{r_i}s_i) = \sum_{i=1}^m (w\widehat{r_i})bs_i$.  Since $w \widehat{r_i}r_i =wr$ for all $i$, we have $w\widehat{r_i} \in F_1^W(wr)$, so that since $bs_i \in bS$ for each $i$, we have $wrx \in F_1^W(wr)bS \subseteq F_1^W(wrbS)$ by Corollary~\ref{cor:hinside}. 
\end{proof}

Applied to the case $M=R$, $S=\{1\}$, we obtain the following:
\begin{cor}\label{cor:recipGring}
Let $A$ be a commutative ring,  let $W$ be a multiplicative submonoid of $A$, and let $\alpha \in W^{-1}A$.  The following are equivalent: 
\begin{enumerate}[(a)]
    \item $\alpha \in R_W(A)$.
    \item There exist $a\in A$ and $b \in W$ such that $\alpha=a/b$ and $a \in [b]^W$.
   \item There exist $a\in A$ and $b \in W$ such that $\alpha=a/b$ and $a \in G^W(b)$.
    \item For any $a\in A$ and $b\in W$ such that $\alpha=a/b$, we have $a\in G^{W}(b)$.
    \item For any $a\in A$ and $b\in W$ such that $\alpha=a/b$, there is some $h\in W$ with $ha \in F_1^W(hb)$.
\end{enumerate}
 Consequently, in the special case where $W$ consists of regular elements of $A$ (e.g. if $A$ is a domain and $0\notin W$), so that $A$ can be identified with its image in $Q = W^{-1}A$, we have \[
R_W(A) = \bigcup_{b\in W} \left([b]^W :_Q b\right).
\]
\end{cor}

 Next, we obtain a description of $G^W_M(S)$ in terms of modules of reciprocals:
\begin{cor}
    Let $A$ be a commutative ring, $W \subseteq A$ a multiplicative submonoid, $M$ an $A$-module and $S \subseteq M$ a subset.  Then $G^W_M(S) = \bigcup_{h\in W} ([hS]^W_M :_A h)$ is the set of elements $x\in M$ such that $x/1 \in R_W^S(M)$.
    In other words, if $\ell: M \ra W^{-1}M$ is the localization map, then \[
    G^W_M(S) = \ell^{-1}(R^S_W(M)).
    \]
\end{cor}

\begin{cor}
Let $A$ be a commutative ring, let $W$ be a multiplicative submonoid of $A$, and let $b \in W$.  Then $G^W(b) = \bigcup_{h \in W} ([hb]^W :_A h)$ is the set of the numerators in $A$ of the elements of $R_W(A)$ that can be written with denominator $b$.  In other words, the smallest $W$-regular $W$-factroid of $A$ containing $b$ is given by $$G^W(b)  = \{a \in A \mid a/b \in R_W(A)\}.$$   Equivalently, for any $a/b \in W^{-1}A$, where $a \in A$ and $b \in W$, one has $a/b \in R_W(A)$ if and only if $a \in G^W(b)$.
\end{cor}

\begin{cor}
Let $A$ be a commutative ring, and let $W$ be a multiplicative submonoid of $A$.   Then the set of all $W$-Egyptian elements of $A$ is the smallest $W$-regular $W$-factroid $G^{W}(1) = \bigcup_{b\in W} \left([b]^W :_A b\right)$ of $A$ containing $1$.  That is, \[
G^W(1) = A \cap R_W(A).
\]
Consequently, the following are equivalent: \begin{enumerate}[(a)]
    \item  $A$ is $W$-Egyptian.
    \item $G^{W}(1) = A$.
    \item $A$ is the only $W$-regular $W$-factroid of $A$ containing $1$.
\end{enumerate}
\end{cor}

Let $\frac{W}{1}$ denote the image of $W$ in $W^{-1}A$.
Since $R_W(A)$ is the set of all $\frac{W}{1}$-Egyptian elements of $W^{-1}A$, and since every $\frac{W}{1}$-factroid of $W^{-1}A$ is $\frac{W}{1}$-regular (as the elements of $\frac{W}{1}$ are all units of $W^{-1}A$), the corollary above yields the following.

\begin{cor}\label{cor:Wrecip1}
Let $A$ be a commutative ring, let $W$ be a multiplicative submonoid of $A$, let $Q = W^{-1}A$, let $\frac{W}{1}$ denote the image of $W$ in $Q$, and let $\frac{1}{W} = \left(\frac{W}{1}\right)^{-1}$ denote the monoid of reciprocals in $Q$ of the elements of $\frac{W}{1}$.  One has 
$$R_W(A) = \Z\langle \tfrac{1}{W} \rangle =  [1]_Q^{\frac{W}{1}} = G_Q^{\frac{W}{1}}(1).$$  Equivalently, $R_W(A)$ is the principal  $\frac{W}{1}$-factroid of $Q$ generated by $1$ (which is $\frac{W}{1}$-regular).  Consequently, the following are equivalent: \begin{enumerate}[(a)]
\item $A$ is $W$-Egyptian.
\item  $Q$ is the only $\frac{W}{1}$-factroid of $Q$ containing $1$.
\item $Q = [1]_Q^{\frac{W}{1}}$.
\end{enumerate}
\end{cor}

\begin{cor}\label{cor:Wrecip2}
Let $A$ be an integral domain with quotient field $K$, so that $\reg(A) = A^\circ = A \setminus \{0\}$.  Then the following are equivalent: \begin{enumerate}[(a)]
    \item $A$ is Egyptian.
    \item $A$ is the only $\reg(A)$-regular $\reg(A)$-factroid of $A$ containing $1$.
    \item $K$ is the only $\reg(A)$-factroid of $K$ containing $1$.
\end{enumerate}
More generally, the Egyptian elements of $A$ (resp., $K$) are precisely the elements of the smallest $\reg(A)$-regular $\reg(A)$-factroid $G^{\reg(A)}_A(1)$  (resp., the smallest $\reg(A)$-factroid
$[1]^{\reg(A)}_K$) of $A$ (resp., $K$) containing $1$.
\end{cor}

\begin{rem}
    Corollaries \ref{cor:Wrecip1} and \ref{cor:Wrecip2} above also follow from Corollary \ref{cor:duality} and Remark \ref{rem:Tregbasics}.
\end{rem}

\begin{example}[Guerrieri]\label{ex:Gnotbracket}
Let $A = \F_2[x,y]$, $W = \reg(A)$, and $f = (x+y^2)(y+x^2)$.  Then $xy \in G^W(f) \setminus [f]_A^W$.

To see this, write $F = [f]_A^W$; we will use that $F$ is a $W$-factroid. Write $g=x+y^2$ and $h=y+x^2$, so that $f=gh$.  Then $g,h \in F$ since $f\in F$, so $(x+y)(1+x+y) = x+y+(x+y)^2 = x+y+x^2+y^2 = f+g \in F$.  Thus, $x+y \in F$, so $y(1+y) = y+y^2 = (x+y)+(x+y^2) = (x+y)+g \in F$; thus, $y\in F$.  Similarly, $x(1+x) = (x+y)+h \in F$, so that $x\in F$.  Then $x+g=y^2 \in F$ and $y+h = x^2 \in F$.  For a set $S$ let $\langle S\rangle$ be the $\F_2$-vector space spanned by $S$.  We have $\langle 1,x,y,x^2,y^2, f\rangle \subseteq F$; to show equality, it suffices to show that $\langle 1,x,y,x^2,y^2, f\rangle$ is itself a $W$-factroid.  For this, note that $\langle 1,x,y,x^2,y^2\rangle$ has 32 elements; one can see by exhaustive search that for any element $u$ of $\langle 1,x,y,x^2,y^2\rangle$, either $u+f$ is irreducible, or $u+f$ is a product of elements of $\langle 1,x,y,x^2,y^2\rangle$.

On the other hand, we have \[
\frac{1+xy}f = \frac 1{xy} + \frac 1{x(y+x^2)} + \frac 1{y(x+y^2)} \in R_W(A).
\]
Hence, by Theorem~\ref{thm:recipG}, $1+xy \in G^W(f)$, so that since $1\in G^W(f)$, we have $xy \in G^W(f)$.
\end{example}

\section{Euclidean domains}\label{sec:Euclidean}

In this section, we study the factroids of any Euclidean domain.  

First we review some preliminaries. A function $\sigma: A \ra \N \cup\{-\infty\}$ on an integral domain $A$ is a \emph{Euclidean function on $A$} \cite[Definition 3.8]{Hun-algbook} if, for all $a,b \in A$ with $b \neq 0$, we have the following: \begin{itemize}
    \item $\sigma(a) = -\infty$ if and only if $a = 0$, 
    \item There exist $q,r \in A$ such that $a=bq+r$ and $\sigma(r)<\sigma(b)$, and
        \item $\sigma(a) \leq \sigma(ab)$.
    \end{itemize}
    If $\sigma$ is a Euclidean function on $A$ and $a \in A$, we say that $\sigma(a)$ is the {\it $\sigma$-degree of $a$}.  Also (see, e.g., \cite[Proposition 2.6]{nme-Euclidean}), for any nonzero $a,b \in A$ one has $\sigma(ab) = \sigma(b)$ if and only if $a$ is a unit, and therefore $\sigma(1)=\sigma(u) < \sigma(a)$ for all units $u$ and all nonzero nonunits $a$.  
    
 A function on $A$ that satisfies all but  the third condition is also, by many authors (e.g. \cite[\S 8.1]{DumFo-algbook}), said to be a ``Euclidean function'' on $A$.   Due to the conflict in terminology (which is allowed to exist because both definitions give rise to the same class of domains \cite{Mo-Euclidean}) we will give a different name to Euclidean functions in this more general sense; we will call such a function a {\it generalized Euclidean function}.  Among the generalized Euclidean functions, those that also satisfy the third condition are said to be \emph{submultiplicative}.  For any generalized Euclidean function $\sigma$, the corresponding function $\rho(a) = \min\{\sigma(ba): b \in A \setminus \{0\}\}$ is a submultiplicative generalized Euclidean function on $A$, i.e. a Euclidean function in the sense of our definition of Eudlidean functions. 

 An integral domain $A$ is said to be \emph{Euclidean} if there is some  generalized Euclidean, or, equivalently, some Euclidean function, on $A$.

\begin{defn}
We say that a Euclidean function $\sigma$ on $A$ is \emph{non-archimedean}  if for all $a,b \in A$, $\sigma(a+b) \leq \max\{\sigma(a), \sigma(b)\}$.
\end{defn}

\begin{example}
The degree function $\deg$ is a non-archimedean Euclidean function on the Euclidean domain $k[x]$, for any field $k$.  However, the Euclidean domain $\Z$ does not admit a non-archimedean Euclidean function. 
\end{example}

Our first result, in a limited sense, characterizes the factroids of any Euclidean domain.  However, it does so more explicitly for those Euclidean domains that admit a non-archimedean Euclidean function.

\begin{thm}\label{prop:normeuclidean}
Let $A$ be a Euclidean domain with a generalized Euclidean function $\sigma: A \to \N \cup \{-\infty\}$. For all $n \in \N \cup \{\pm \infty\}$, let \[
C_n := \{a \in A: \sigma(a) \leq n\}.
\]
\begin{enumerate}
    \item Let $F\neq A$ be a factroid of $A$. Then there is some $n \in \sigma(A)$ such that for any $x\in A$ with $\sigma(x)=n$, we have $F=C_n = [x]_A$. 
    \item $\sigma$ is a Euclidean function on $A$ if and only if all of the sets $C_n$ are $\reg(A)$-saturated.
  \item $\sigma$ is non-archimedean if and only if the sets $C_n$, for $n \in \sigma(A) \cup \{\infty\}$, are all additive subgroups of $A$.
   \end{enumerate}
    Hence, $\sigma$ is \begin{itemize}
        \item a Euclidean function if and only if the factroids of $A$ are precisely the sets $C_n$ that are additive subgroups of $A$, and
        \item a non-archimedean Euclidean function if and only if the factroids of $A$ are all the sets $C_n$.
    \end{itemize}
\end{thm}

\begin{proof}
First, note that for any $n \in \N \cup \{\pm \infty\}$, we have $C_n = C_s$, where $s = \min\{t \in \sigma(A) \cup \{\pm \infty\} \mid t \geq n\}$.
 
   To prove (1), it suffices to show that $C_{n} \subseteq F$ for any factroid $F$ of $A$ containing some element of $A$ of  $\sigma$-degree at least $n$, where $n \in \N \cup \{-\infty\}$.   We prove this by induction on $n$, the case $n=-\infty$ being trivial.   Let $n\geq 0$, and suppose that for all $\ell<n$ with $\ell \in \N \cup \{-\infty\}$, we have $C_\ell \subseteq F$ whenever $F$ contains some element of $\sigma$-degree at least $\ell$.  Let $f\in F$ of $\sigma$-degree at least $n$.  Let $0 \neq g\in F$ with $\sigma(g) \leq n$. Then there exist $q,r \in A$ with $f=qg+r$ and $\sigma(r)<\sigma(g)$.  Since $f\neq r$ (as they have different $\sigma$-degrees), we have $q \neq 0$. Since $\sigma(r)<\sigma(g)=n \leq \sigma(f)$ and $f\in F$, it follows by inductive hypothesis that $C_{\sigma(r)} \subseteq F$, whence $r\in F$. 
 Thus, $qg = f-r \in F$, so since $q\neq 0$ and $F$ is a factroid, we have $g\in F$. 

To prove (2), suppose that $\sigma$ is a Euclidean function, and let $n \in \sigma(A) \cup \{\infty\}$.  Let $a \in A$ and $b \in \reg(A)$ with $ba \in C_n$.  Then $\sigma(a) \leq \sigma (ba) \leq n$, so that $a \in C_n$.  Thus the sets $C_n$ are all $\reg(A)$-saturated.  Conversely, suppose that all of the sets $C_n$ are $\reg(A)$-saturated.  Let $a,b \in A$ with $b \neq 0$.  Let $n = \sigma(ab)$.  Since $C_n$ is $\reg(A)$-saturated and $ab \in C_n$, one has $a \in C_n$, and therefore $\sigma(a) \leq n = \sigma(ab)$.  Therefore $\sigma$ is a Euclidean function on $A$.  

To prove (3), suppose that $\sigma$ is non-archimedean. For any $n$, let $a,b \in C_n$.  Then $\sigma(a\pm b) \leq \max\{\sigma(a), \sigma(b)\} \leq n$, so $a \pm b \in C_n$. Thus, $C_n$ is an additive subgroup of $A$.  Conversely,  suppose that each $C_n$ is an additive subgroup of $A$.  Let $a,b \in A$, and suppose without loss of generality that $\sigma(a) \leq n: = \sigma(b)$.  Then $a,b \in C_n$, whence $a+b \in C_n$, so that $\sigma(a+b) \leq n = \max(\sigma(a),\sigma(b))$.
\end{proof}

Proposition \ref{ex:polysimple}, whose proof we postponed, is an immediate corollary of Theorem \ref{prop:normeuclidean} above. The following proposition is a partial converse of Proposition~\ref{ex:polysimple}.  We define $2^\sigma$ to be the function $a \mapsto 2^{\sigma(a)}$.

\begin{prop}\label{prop:non-arch2}
    For any Euclidean function $\sigma$ on an integral domain $A$, the following conditions are equivalent:
    \begin{enumerate}[(a)]
    \item $\sigma$ is non-archimedean and satifies the condition $\sigma(ab) = \sigma(a)+\sigma(b)$ for all $a,b \in A$.
    \item The function $2^\sigma: A \to \N$ is a non-archimedean norm on the domain $A$.
    \item The function $2^\sigma$ is a norm on the domain  $A$, and  the prime subring of $A$ is contained in $\{a \in A: \sigma(a) \leq \sigma(1)\} = A^\times \cup \{0\}$.
    \item The function $2^\sigma$ is a norm on the domain $A$, and $A$ contains a field.
        \item   The factroids of $A$ are the sets
 \[
 C_n := \{a \in A: \sigma(a) \leq n\}
 \] for all $n \in \sigma(A) \cup \{\infty\}$ (which are distinct for all $n \in \sigma(A)$), and the  level sets $$D_n := \{a \in A: \sigma(a) = n\}$$ 
 satisfy $$D_m D_n \subseteq D_{m+n}$$
 for all  $m,n \in \sigma(A)$.
    \end{enumerate}
    Moreover, a domain $A$ admits a Euclidean function satisfying the equivalent conditions above if and only if $A$ is a field or $A \cong k[x]$ for some field $k$.
 \end{prop}

 \begin{proof} 
     The equivalence of  conditions (a)--(c) follows from standard facts about norms on integral domains, and the equivalence of (c) and (d) follows from the fact that $\{a \in A\mid \sigma(a) \leq \sigma(1)\} = A^\times \cup \{0\}$.   The equivalence of conditions (a) and (e) follows from Theorem \ref{prop:normeuclidean}.
    Finally, suppose that $\sigma$ satisfies all of the given conditions.  Then the function $\delta = 2^\sigma$ is a Euclidean function on $A$ in the sense of \cite[Section 2.15]{Jac-BA1e2}, and therefore,  by Exercise  6 of that section, $A$ is a field or $A \cong k[x]$ for some field $k$.  The converse is clear.
 \end{proof}

In fact, the following stronger result is possible and is due to user527492 on MathOverflow.  See \url{https://mathoverflow.net/questions/485986/non-archimedean-euclidean-domains/486062#486062}. 
\begin{prop}[user527492 on MathOverflow]\label{prop:kx}
 An integral domain $A$ admits a non-archimedean Euclidean function on $A$ if and only if $A = k$ or $A \cong k[x]$, where $k = A^\times \cup \{0\}$ is a field.
\end{prop}

\begin{proof} Suppose that $A \neq k$, so that $A$ has an nonzero nonunit $t$ of minimal $\sigma$-degree.
Let $\phi \colon k[x] \to A$ be the unique $k$-algebra map sending $x \mapsto t$. We first show that $\phi$ is surjective. Note that $k$ is in the image of $\phi$ by construction, so it suffices to show that $A \setminus k$ is also in the image of $\phi$. Suppose, for the sake of contradiction, that the image of $\phi$ does not contain $A\setminus k$. Then there exists an element $y \in A \setminus k$ not in the image of $\phi$ of minimal $\sigma$-degree. Since $\sigma$ is Euclidean, there exist $q, r \in A$ such that $y = qt + r$, where $\sigma(r) < \sigma(t)$.  This implies that $r \in k$. If $q = 0$, then $y = r \in k$, which is a contradiction. Therefore, since $\sigma$ is a non-archimedean Euclidean function, 
$$
\sigma(q) < \sigma(qt) = \sigma(y-r) \leq \max(\sigma(y), \sigma(-r)) = \sigma(y).
$$
By minimality, $q = \phi(f)$ for some $f \in k[x]$, which implies $y = \phi(fx+r)$, which is a contradiction.

We now show that $\phi$ is injective. Suppose, for the sake of contradiction, that $\phi(f) = 0$, where $f = a_nx^n + \cdots + a_1x +a_0  \neq 0$.  Let $k$ be least such that $a_k \neq 0$.  Then $-a_k = \phi(-a_k) = (\phi(a_n)t^{n-k-1} + \cdots + \phi(a_{k+1}))t$, implying that $t$ is a unit, which is a contradiction. Hence $\phi$ is also injective, and therefore an isomorphism.
\end{proof}

\begin{problem}
    Does there exist a Euclidean domain that admits a non-archimedean generalized Euclidean function  and yet is neither a field nor isomorphic to $k[x]$ for some field $k$?
\end{problem}

\section{Unit-additive rings and sublocalizable rings}\label{sec:unitadditive}

 In this section, in which all rings are assumed to be commutative, we show that the recently discovered notion of \emph{unit-additivity} is intimately connected to the subring $[1]^{A^\circ}_A$ of a ring $A$.  We also introduce a common generalization of the local rings an the unit-additive rings that we call the \emph{sublocalizable rings}, and we show, analogously, their intimate connection to the subring $[1]^{A^\times}_A = \Z \langle A^\times \rangle$ of a  ring $A$ generated by the units of $A$.

\begin{defn}[{\cite{nmeSh-unitadd}}]\label{def:uadd}
A  ring $A$    is \emph{unit-additive}  if for any units $u,v \in A^\times$, $u+v$ is either a unit or nilpotent.
\end{defn}

\begin{example}
In \cite{nmeSh-unitadd}, the second named author and Jay Shapiro provided numerous examples to show that unit-additivity is a very natural condition.  For instance: \begin{itemize}
\item If $A \subseteq B$ is an integral extension of  rings and $B$ is unit-additive, then so is $A$ \cite[Proposition 2.8]{nmeSh-unitadd}.
\item Any Boolean ring is unit-additive \cite[Example 2.14]{nmeSh-unitadd}.
\item Unit-additivity characterizes a kind of simplicity in large classes of Euclidean domains.  Let $A$ be an Euclidean domain that is not a field and is  finitely generated as a $k$-algebra, where $k$ is a maximal subfield of $A$.  If $k=\F_2$ or $k$ is algebraically closed, then $A$ is unit-additive  if and only if $A \cong k[x]$ \cite[Theorem 2.15]{nmeSh-unitadd}.
\item Unit-additivity can capture properties of monoids.  Indeed, let $A$ be a nonzero  ring of characteristic zero, and let $M$ be a cancellative commutative  monoid.  Then the monoid algebra $A[M]$ is unit-additive if and only if $A$ is unit-additive and $M$ is torsion-free with no nontrivial invertible elements (see \cite[Theorem 3.6]{nmeSh-unitadd}).  
\item Unit-additivity is related to the fundamental theorem of algebra.  In fact, let $X$ be an irreducible variety over an algebraically closed field $k$, and let $A$ be its coordinate ring.  Then $A$ is unit-additive if and only if every polynomial mapping $X \ra \A^1_k$ has a root (see \cite[Remark 4.4]{nmeSh-unitadd}).
\end{itemize}
\end{example}

\begin{example} Let $A$ be a Euclidean domain  with a Euclidean function $\sigma$.   Then $C_{\sigma(1)} = A^\times \cup \{0\} = A^\times \cup \sqrt{0}$ is a multiplicative submonoid of $A$. Moreover,  by Theorem~\ref{prop:normeuclidean}(3) and \cite[Proposition 2.1]{nmeSh-unitadd}, the following are equivalent: \begin{enumerate}[(a)]
    \item $A$ is unit-additive.
    \item $C_{\sigma(1)}$ is additively closed.
    \item $C_{\sigma(1)}$ is a factroid of $A$.
    \item $C_{\sigma(1)}$ is a subring of $A$.
    \item  $C_{\sigma(1)}$ is a subfield of $A$.
\end{enumerate}
\end{example}

For any  ring $A$, let $J(A)$ denote the Jacobson radical of $A$.

\begin{lemma}\label{lem:jacobson}
Let $A$ be a  ring.  Then $A^\times \cup J(A) \subseteq \Z \langle A^\times \rangle = [1]^{A^\times}_A$.  Consequently, the units of $A$ are the same as those of its subring $R := [1]^{A^\times}_A$, and $J(A) \cap R \subseteq J (R)$.
\end{lemma}

\begin{proof}
The fact that $\Z \langle A^\times \rangle =[1]^{A^\times}_A$ follows from Corollary \ref{cor:WcapM}.  Let $a\in J(A)$.  Then $1+a$ is a unit, so since $A^\times \subseteq R$, it follows that $1+a \in R$.  Since also $1\in R$, we have $a=(1+a)-1 \in R$.  Now, let $a \in J(A) \cap R$.  Then $1+ra \in A^\times = R^\times$ for all $r \in A$, hence for all $r \in R$, whence $a \in J(R)$.
\end{proof}

\begin{prop}\label{pr:uaddbracket1}
    Let $A$ be a  ring.  Then $A^\times \cup \sqrt{(0)} \subseteq [1]^{A^\circ}_A$; equivalently, the units and nilpotent elements of $A$ are the same as those of its subring $[1]^{A^\circ}_A$.  Moreover, the inclusion is an equality if and only if $A$ is unit-additive,  if and only if $[1]^{A^\circ}_A$ is unit-additive,  if and only if the inclusions $$A^\times \cup \sqrt{(0)} \subseteq  A^\times \cup J(A)  \subseteq[1]^{A^\times}_A  \subseteq [1]^{\reg(A)}
    _A \subseteq [1]^{A^\circ}_A.$$  are equalities.
\end{prop}

\begin{proof}
The given inclusions hold because of  Lemma~\ref{lem:jacobson}.   Therefore, if equalities hold, then $A^\times \cup \sqrt{(0)}$ is a subring of $A$, whence by   \cite[Proposition 2.1]{nmeSh-unitadd}  $A$ is unit-additive.  Suppose, conversely, that $A$ is unit-additive.  Then $B = A^\times \cup \sqrt{(0)}$ is a subring of $A$ by \cite[Proposition 2.1]{nmeSh-unitadd}, so we only need to show that $B$ is an $A^\circ$-saturated subset of $A$.  So let $a\in A^\circ$ and $x\in A$ with $ax \in B$. Then either $ax\in A^\times$ or $ax \in \sqrt{(0)}$.  If $ax\in A^\times$, then $x$ is a unit, so $x\in A^\times \subseteq B$. Finally, suppose $ax\in \sqrt{(0)}$.  Then for any minimal prime $\p$ of $A$, we have $ax\in \p$, but $a\notin \p$, so $x\in \p$.  Since $\p$ was arbitrary, $x\in \sqrt{(0)} \subseteq B$.
\end{proof}

\begin{example} Although, for any integral domain $D$, its subring $R := [1]^{\reg(D)}_D$ contains $[1]^{D^\times}_D = \Z \langle D^\times \rangle$, equality need not hold.  For instance, let $D = \Z[x, 2/x] \subseteq \Q(x)$. Since the only units of $D$ are $1$ and $-1$, the subring  of $D$ generated by $D^\times$ is just $\Z$.  But since $x (2/x) = 2 \in R$, one has $x, 2/x \in R$, and therefore $R = D$.
\end{example}

\begin{construction}[{\cite[Definition 6.4]{nmeSh-unitadd}}]\label{cons:satadd}
Let $A$ be an integral domain.  For each $i\geq 0$, define multiplicatively closed sets $V_i$ and $W_i$ inductively as follows: $W_0 = \{1\}$, $V_0 = A^\times$.  For each $i\geq 1$, we let $W_i$ be the set of nonzero sums of elements of $V_{i-1}$, which by \cite[Lemma 6.3]{nmeSh-unitadd} is a multiplicatively closed set.  We then let $V_i$ be the $\reg(A)$-saturation of $W_i$. Then set $U := \bigcup_i W_i = \bigcup_i V_i$.  By \cite[Proposition 6.2 and Theorem 8.1]{nmeSh-unitadd}, $U$ is a saturated multiplicatively closed set that is closed under nonzero sums.  By \cite[Proposition 8.4]{nmeSh-unitadd}, $U^{-1}A$ is the unique smallest overring of $A$ that is unit-additive, called the \emph{unit-additive closure} of $A$ and written $A^{\operatorname{ua}}$ (see \cite[Definition 8.2]{nmeSh-unitadd}). 
\end{construction}

\begin{thm}
Let $A$ be an integral domain, with fraction field $K$.  Let $A^{\operatorname{ua}}$ and $U$ be as in Construction~\ref{cons:satadd}.  Then $[1]^{\reg(A)}_A = [1]^{A^\circ}_A  = \Z\langle U \rangle = U \cup \{0\}$, and therefore $A^{\operatorname{ua}} = ([1]^{\reg(A)}_A \setminus \{0\})^{-1}A$.  Moreover, one has $[1]^{\reg(A)}_A=A$ if and only if $K$ is the only unit-additive overring of $A$.
\end{thm}

\begin{proof}
    First we verify that $\Z\langle U \rangle = U \cup \{0\}$.  It is obvious that ``$\supseteq$'' holds.  For the opposite containment, let $0\neq x \in \Z\langle U \rangle$.  Then $x= \sum_{j=1}^n u_j$ for some $u_j \in U$.  But then, using the notation of Construction~\ref{cons:satadd}, there is some $i$ such that $u_1, \ldots, u_n \in V_i$.  Hence, $x \in W_{i+1} \cup \{0\} \subseteq U \cup \{0\}$.

    To show that $[1]^{\reg(A)}_A \subseteq U \cup \{0\}$, by Corollary~\ref{cor:bracket1} it suffices to show that the ring $U \cup \{0\}$ is a $\reg(A)$-saturated subring of $A$.  For this, let $a,x\in A$ with $ax \in U \cup \{0\}$ and $a\in \reg(A)$.  If $ax=0$, then since $A$ is a domain and $a\neq 0$, we have $x=0$.  Otherwise, $ax\in U$, so there is some $i$ with $ax \in W_i$.  But then by definition $x\in V_i \subseteq U$.  In either case $x\in U\cup \{0\}$, so this subring of $A$ is $\reg(A)$-saturated.

    For the opposite containment, we show by induction on $i$ that $V_i \subseteq [1]^{\reg(A)}_A$ for all $i\geq 0$, the case $i=0$ being trivial.  So let $i\geq 0$ and assume $V_i \subseteq [1]^{\reg(A)}_A$. Then any sum of elements of $V_i$ is in $[1]^{\reg(A)}_A$, so that $W_{i+1} \subseteq [1]^{\reg(A)}_A$.  Now let $x\in V_{i+1}$.  Then there is some $0\neq a\in A$ with $ax \in W_{i+1} \subseteq [1]^{\reg(A)}_A$.  Then since this is a $\reg(A)$-factroid, it follows that $x\in [1]^{\reg(A)}_A$, completing the induction.  Thus, $U \cup \{0\} = \left(\bigcup_{i\geq 0} V_i\right) \cup \{0\} \subseteq [1]^{\reg(A)}_A$.

    For the final statement, suppose $[1]^{\reg(A)}_A=A$.  Then every nonzero element of $A$ is in $U$, so $K = (A \setminus \{0\})^{-1}A = U^{-1}A = A^{\operatorname{ua}}$.  Thus, $K$ is the only overring of $A$ that is unit-additive.  Conversely suppose $K$ is the only unit-additive overring of $A$ that is unit-additive, so that $U^{-1}A = A^{\operatorname{ua}}=K$.  Let $0\neq a\in A$.  Then $\frac 1a \in K = U^{-1}A$, so there exist $b\in A$ and $u\in U$ with $\frac 1a = \frac bu$ as elements of $K$.  Thus, $ba = u \in U$.  In particular, $ba\neq 0$ since $0\notin U$, so $b\neq 0$.  Since $U$ is $\reg(A)$-saturated and $ba \in U$, we have $a\in U$.  Thus, $A = U \cup \{0\} = [1]^{\reg(A)}_A$.
\end{proof}

The definition of unit-additive rings inspires the following generalization. 

\begin{defn}
     A   ring $A$ is \emph{sublocalizable} if   it is nontrivial and $A^\times + A^\times \subseteq  A^\times \cup J(A)$.
\end{defn}

The following proposition generalizes \cite[Proposition 2.1]{nmeSh-unitadd}   and gives several equivalent conditions for a ring to be sublocalizable.  In particular, the term is motivated by the equivalent condition (i).

\begin{prop}[{cf.\ \cite[Proposition 2.1]{nmeSh-unitadd}}]\label{nmeSh-unitadd}
    Let $A$ be a nontrivial  ring.  The following conditions are equivalent: 
    \begin{enumerate}[(a)]
        \item $A$ is sublocalizable.
          \item $1 + A^\times \subseteq A^\times \cup J(A)$.
          \item $A^\times \cup J(A)$ is a subring of $A$.
     \item There exists a proper ideal $I$ of $A$ such that $A^\times \cup I$ is a subring of $A$.
        \item $J(A)$ is the unique proper ideal $I$ of $A$ such that $A^\times \cup I$ is a subring of $A$.
                \item $B\setminus A^\times$ is an ideal of $A$ for some subring $B$ of $A$ containing $A^\times$ (in which case $B = A^\times \cup J(A)$ and $B\setminus A^\times = J(A)$ are unique).
                \item $A^\times \cup J(A)$ is a local subring of $A$ (with maximal ideal $J(A)$ and having the same units as $A$).
        \item $A$ has a (unique) local subring with maximal ideal $J(A)$ and having the same units as $A$.
        \item $A$ has a (unique) local subring whose maximal ideal is also an ideal of $A$ and whose units are the same as those of $A$.
        \item $A$ has a (unique) local subring $B$ with $B^\times = A^\times$ and $J(B) = J(A)$.
    \item $A/J(A)$ is sublocalizable.
        \item $A/J(A)$ is unit-additive.
    \end{enumerate}
\end{prop}

\begin{proof}
Clearly (a) implies (b).  Suppose that  (b) holds. Then
$$A^\times +A^\times = A^\times(1+A^\times) \subseteq A^\times( A^\times  \cup J(A) ) = A^\times  \cup J(A),$$ whence (a) holds.   Therefore (a) and (b) are equivalent.

Note that, for any ideal $I$ of $A$, one has $u+I \subseteq A^\times$ for some unit $u$ of $A$  if and only if $A^\times +I =  A^\times$, if and only if $I \subseteq J(A)$.  Suppose that (a) holds.  Then one has
\begin{align*}
    (A^\times \cup J(A))+(A^\times \cup J(A))  
& = (A^\times + A^\times) \cup (A^\times +J(A)) \cup (J(A)+J(A)) \\ & \subseteq 
(A^\times \cup J(A)) \cup A^\times \cup J(A) \\
& = A^\times \cup J(A)
\end{align*} and
$$(A^\times \cup J(A))(A^\times \cup J(A)) = A^\times A^\times \cup A^\times J(A) \cup J(A)J(A) = A^\times \cup J(A).$$
It follows that conditions (a) and (c) are equivalent.

Suppose that $A^\times \cup I$ is a subring of $I$ for some proper ideal $I$ of $A$. We show that $I = J(A)$.  Note first that $1 + I \subseteq A^\times \cup I$
and $1 + I$ is disjoint from $I$ since $I$ is proper.  Therefore $1 + I \subseteq A^\times$, whence $I \subseteq J(A)$.  To prove the reverse containment, note that
$- 1 +J(A) \subseteq A^\times$ and therefore
$J(A) \subseteq 1 + A^\times \subseteq A^\times \cup I$,
whence $J(A) \subseteq I$ since $J(A)$ is disjoint from $A^\times$.  It follows, then, that $I = J(A)$.  All told, this shows that conditions (a)--(j) are equivalent.

Finally, one has $J(A/J(A)) = (0)$ and therefore $\operatorname{nilrad}(A/J(A)) = (0)$, while also $(A/J(A))^\times = \{a + J(A): a \in A^\times \} $ and $(A/\operatorname{nilrad}(A))^\times = \{a+ \operatorname{nilrad}(A): a \in A^\times\}$, from which  it follows that (j)--(l) are equivalent.
\end{proof}

If a  ring $A$ is sublocalizable, then we say that the \emph{sublocalization of $A$} is the unique local subring of $A$, namely  $A^\times \cup J(A)$, whose maximal ideal is also an ideal of $A$ and whose units are the same as those of $A$.
Since any local ring is generated by its group of units, if $A$ is sublocalizable, then its sublocalization $B$ is equal to $\Z\langle B^\times \rangle = \Z \langle A^\times \rangle = [1]^{A^\times}_A$.  More generally, a ring  $A$ having a local subring containing the units of $A$ has a unique such subring, namely $\Z \langle A^\times \rangle$.  Thus, we have the following.

\begin{cor}\label{cor:uniquesublocalization}
Let $A$ be a sublocalizable ring, with sublocalization $B$. Then $B$ is the unique local subring of $A$ that contains $A^\times$.  Moreover, $B$ is the smallest subring $\Z \langle A^\times \rangle = [1]^{A^\times}_A$ of $A$ containing $A^\times$.
\end{cor}

\begin{example}\label{ex:weaksub}
A  ring  $A$ having a (unique) local subring sharing the units of $A$ need not be sublocalizable.   For instance, let $B$ be any reduced local ring 
other than a field, and let $A = B[x]$.  By \cite[Exercises 2 and 4 of Chapter 1]{AtMac-ICA},  $A^\times = B^\times$ and $J(A)=(0)$.  But $A^\times \cup J(A) = B^\times \cup (0)$ is not additively closed, since for any nonzero nonunit $b$ of $B$, we have $-1, 1+b \in B^\times = A^\times$, so $b \in (A^\times + A^\times) \setminus (A^\times \cup J(A))$.  Thus, $A$ is  not sublocalizable.
\end{example}

\begin{cor}\label{cor:jac}
Let $I$ be proper ideal of a   ring $A$.  The following conditions are equivalent:
    \begin{enumerate}[(a)]
    \item $A$ is sublocalizable and $I = J(A)$.
    \item $A$ is sublocalizable, with sublocalization $A^\times \cup I$.
      \item $A^\times \cup I$ is a subring of $A$.
     \item $A^\times \cup I$ is a local subring of $A$ (with maximal ideal $I$ and having the same units as $A$).
     \item $A^\times + A^\times \subseteq A^\times \cup I$ and $I \subseteq J(A)$.
    \item  $A/I$ is unit-additive and reduced, and $I \subseteq J(A)$.
   \end{enumerate}
\end{cor}

\begin{cor}
    A ring $A$ is unit-additive if and only if $A$ is sublocalizable and $J(A) = \sqrt{(0)}$.
\end{cor}

\begin{rem}
    A ring $A$ such that $J(A) = \sqrt{(0)}$ is said to be an \emph{NJ ring} \cite{JRW-NJ}.
\end{rem}

\begin{rem}
A ring $A$ is a local ring if and only if $A \setminus A^\times$ is an ideal of $A$, in which case $A \setminus A^\times = J(A)$ is the unique maximal ideal of $A$.  Thus, a local ring is equivalently a sublocalizable ring $A$ with sublocalization $A$, or equivalently a  ring $A$ such that $A \setminus A^\times = J(A)$.
Therefore, the sublocalizable rings can be thought of as a common generalization of  the local rings and the unit-additive rings.

A  ring  is $A$ both is local and unit-additive if and only if  $A \setminus A^\times = \sqrt{(0)}$, if and only if $A$ has a unique prime ideal, if and only if $A$ is local and zero dimensional.  Consequently, a   ring $A$ is unit-additive if and only if it has a unique zero dimensional local subring whose (unique) prime ideal is an ideal of $A$ and whose units are the same as those of $A$, if and only if it is sublocalizable with zero dimensional sublocalization.  Thus, one could say that the sublocalizable rings are to the local rings as the unit-additive rings are to the zero dimensional local rings.

     A  ring $A$ is \emph{semiprimitive} if $J(A) = (0)$.  Since the nilradical is always contained in the Jabobson radical, every semiprimitive ring is reduced.  A  ring $A$ is semiprimitive and  sublocalizable if and only if $A^\times + A^\times \subseteq A^\times \cup \{0\}$, if and only if $A$ is reduced and unit-additive, if and only if $A$ is sublocalizable with sublocalization  a field, if and only if $A$ has a subfield whose units are the same as those of $A$.  Thus, among the semiprimitive rings, the sublocalizable and the unit-additive rings coincide.  
\end{rem}

Proposition \ref{pr:uaddbracket1} generalizes as follows.

\begin{cor}\label{lem:Lunitadgen}
Let $A$ be a  ring and $I$ a proper ideal of $A$, and let $T$ be any subset of $A$ with $A^\times \subseteq T \subseteq W(I)$.   Then the subring $[1]^T_A$ of $A$   contains $I$ as an ideal and has the same units as $A$.  In particular, $A^\times \cup I \subseteq [1]^T_A$. Moreover, the following conditions are equivalent:
\begin{enumerate}[(a)]
    \item $A$ is sublocalizable and $I = J(A)$.
    \item The inclusion $A^\times \cup I  \subseteq [1]^T_A$ is an equality.
    \item The ring $[1]^T_A$ is sublocalizable  and $I = J([1]^T_A)$.
    \item The ring  $[1]^T_A$ is local, with maximal ideal $I$.
    \item $[1]^T_A \setminus I = A^\times$.
        \item $[1]^T_A \setminus I = ([1]^T_A)^\times$.
    \item $A$ has a (unique) local subring (namely $[1]^{T}_A$) having the same units as $A$ and  maximal ideal $I$.
        \item $A$ is sublocalizable,  with sublocalization $[1]^{T}_A$, which in turn has maximal ideal $I$.
\end{enumerate}
If the equivalent conditions  above hold, then $I = \sqrt[A]{I} = J(A)$,  $[1]^{A^\times}_A = [1]^{T}_A  = [1]^{W(J(A))}_A$, and $W_A([1]^{A^\times}_A) = W_A(J(A)) = W(I)$ is the set of all elements of $A$ that do not lie in any prime ideal of $A$ that is minimal over $J(A)$.
\end{cor}

\begin{proof}
Since $A^\times \subseteq T$, one has $A^\times \subseteq  [1]^{A^\times}_A \subseteq [1]^T_A$, and therefore $[1]^T_A$ has the same units as $A$. Moreover,
one has $I \subseteq [1]^T_A I \subseteq [I]^T_A = I$ since $I$ is a $T$-factroid of $A$, and therefore $I$ is an ideal of the ring $[1]^T_A$.  

Suppose that (a) holds, whence, by Corollary~\ref{cor:jac}, $B := A^\times \cup I$ is a subring of $A$.   Moreover, both $A^\times$ and $I$ are $T$-saturated subsets of $A$, and therefore so is $B$.  It follows, then, that $B$ is a $T$-factroid of $A$ containing $1$ and contained in $[1]^T_A$, whence (b) holds.  Conversely, if (b) holds, then (a) holds, again by Corollary \ref{cor:jac}.  Moreover, the  remaining equivalences follow, along with the equalities $[1]^{A^\times}_A = [1]^{T}_A  = [1]^{W(J(A))}_A$ and $I= \sqrt{I} = J(A)$, by Proposition~\ref{nmeSh-unitadd}.  Finally, the equality $W_A(B) = W_A(I)$, where $B =[1]^{A^\times}_A$, readily follows from the facts that $B = B^\times \cup I$ is a disjoint union, $B^\times = A^\times$, and $ax \in A^\times$ if and only if $a,x \in  A^\times$, for any $a,x \in A^\times$.
\end{proof}

 \begin{example}
The condition that the rings $\Z \langle A^\times \rangle =   [1]^{A^\times}_A \subseteq   [1]^{\reg(A)}_A$ and $[1]^{W(J(A))}_A$ are all equal and local does not guarantee  that $A$ is sublocalizable, even when $A$ is an integral domain. To see this, let $B$ be any local domain and  $A = B[x]$, as in Example~\ref{ex:weaksub}.   We have $A^\times=B^\times$ and $J(A) = (0)$, so that $W(J(A)) = \reg(A)$  and $A^\times = B^\times$, and therefore $B = \Z \langle B^\times \rangle = \Z \langle A^\times \rangle$. Moreover, note that $B$ is $\reg(A)$-saturated.  This follows by consideration of degrees in the variable $x$, since if $f,g \in A$ with $0\neq f$ and $fg \in B$, then $0 = \deg(fg) = \deg(f) + \deg(g)$, whence $\deg(g)=0$ so that $g\in B$.  Thus, we have $[1]^{W(J(A))}_A = [1]^{\reg(A)}_A = B$.  But by Example~\ref{ex:weaksub}, $A$ is not sublocalizable.
\end{example}

The following corollary generalizes  to sublocalizable rings known results on how the local and unit-additive properties behave with respect to integral extensions (see \cite[Lemma 2 of Section 9]{Mats} and \cite[Proposition 2.8]{nmeSh-unitadd}).

\begin{cor}[{cf.\ \cite[Proposition 2.8]{nmeSh-unitadd}}]
     Let $A \subseteq B$ be an integral extension of rings.
     Suppose that $B$ is sublocalizable.  Then $A$ is sublocalizable, with unique local subring $[1]^{A^\times}_A =  [1]^{B^\times}_B \cap A$ that has maximal ideal $J(A) =  J(B) \cap A$ and units $A^\times = B^\times \cap A$.
\end{cor}

 \begin{proof}
The equalities $J(A) = J(B)  \cap A$ and  $A^\times =  B^\times \cap A$ are well-known properties of integral extensions (see \cite[Exercise 5 of Chapter 5]{AtMac-ICA}).  Suppose that $B$ is sublocalizable.  Let $u, w \in A^\times$.  Then either $u+w \in B^\times \cap A = A^\times$ or $u+w \in J(B) \cap A = J(A)$, so that $A$ is sublocalizable.  Consequently, one has $[1]^{A^\times}_A = A^\times \cup J(A)  =    (B^\times \cap  A) \cup (J(B) \cap A )   = (B^\times \cup J(B)) \cap A =  [1]^{B^\times}_B \cap A$.  This completes the proof.
\end{proof}

\begin{problem}
Undertake a study of sublocalizable rings, e.g., generalize the results of \cite{nmeSh-unitadd} to the extent possible.  
\end{problem}

\section{Factroids of submodules, quotients, and preservation along homomorphisms}\label{sec:maps}

A natural question to ask is how factroids of modules transfer or are preserved under ring and module homomorphisms.  In general, preimages are more well-behaved than images.

\begin{prop}\label{pr:ringpreimage}
Let $\phi: A \ra B$ be a ring homomorphism and $T \subseteq B$ a subset. 
\begin{enumerate}
    \item If $T$ is a multiplicative subset (resp., submonoid) of $B$, then $\varphi^{-1}(T)$ is a multiplicative subset (resp., submonoid) of $A$.
    \item
For any $T$-factroid $F$ of $B$, $\phi^{-1}(F)$ is a $\phi^{-1}(T)$-factroid of $A$.
    \item For any additive subgroup $F$ of $B$, one has $\varphi^{-1}(W(F))\subseteq W(\varphi^{-1}(F))$.
    \item For any $B$-module $N$ and any $T$-factroid $G$ of $N$ (as a $B$-module), $G$ is a $\phi^{-1}(T)$-factroid of $N$ as an $A$-module of $N$ by restriction of scalars along $\varphi$.
    \item For any $B$-module $N$ and any additive subgroup $G$ of $N$,  one has $\varphi^{-1}(W_N(G)) \subseteq  W_{{}_\varphi N}(G)$, where ${}_\varphi N$ denotes the $A$-module $N$ obtained by restriction of scalars along $\varphi$.
\end{enumerate}
\end{prop}

\begin{proof}
Statement (1) is clear.

Let $F$ be a $T$-factroid of $B$.  It is clear that $\phi^{-1}(F)$ is an additive subgroup of $A$. Let $r,a \in A$ with $r\in \phi^{-1}(T)$ and $ra \in \phi^{-1}(F)$.  Then $\phi(r)\phi(a) = \phi(ra) \in F$, so since $\phi(r) \in T$,  one has $\phi(a) \in F$, that is, $a\in \phi^{-1}(F)$.  This proves statement (2).  Statement (3) then follows by applying statement (2) to the set $T = W(F)$.

To prove statement (4), let $r\in \phi^{-1}(T)$ and $x \in N$ with $rx \in G$.  By definition, $rx = \phi(r)x$, so since $\phi(r) \in T$, we have $x\in G$.    This proves (4), and statement (5) then follows by applying (4) to the set $T = W_N(G)$.
\end{proof}

Below is the analogue of Proposition \ref{pr:ringpreimage} for module homomorphisms.

\begin{prop}\label{pr:modulepullback}
Let $A$ be a ring  and $h: M \ra N$ an $A$-module homomorphism.  Then, for any subset $T$ of $A$ and for any $T$-factroid $G$ of $N$, $h^{-1}(G)$ is a $T$-factroid of $M$. Equivalently, for any additive subgroup $G$ of $N$, one has $W_N(G) \subseteq W_M(h^{-1}(G))$.
\end{prop}

\begin{proof}
Clearly $\varphi^{-1}(G)$ is an additive subgroup of $M$.  Let $a\in T$ and $x\in M$ such that $ax \in h^{-1}(G)$.  Then $ah(x) = h(ax) \in G$, so since $G$ is a $T$-factroid, $h(x) \in G$, whence $x\in h^{-1}(G)$.  Thus, $\varphi^{-1}(G)$ is a $T$-factroid of $M$.
\end{proof}

\begin{prop}\label{pr:VLinWM}
Let $\varphi: A \ra B$ be a ring homomorphism, $T \subseteq B$ a subset, $M$ a $B$-module, $L$ an $A$-submodule of the $A$-module $M$ formed by restriction of scalars along $\varphi$, and $S \subseteq L$ a subset. Then 
$[S]_L^{\varphi^{-1}(T)}\subseteq [S]_M^T$.
\end{prop}

\begin{proof}
Set $U := \varphi^{-1}(T)$.  Let $G$ be a $T$-factroid of $M$ containing $S$, where $M$ is thought of as a $B$-module.  Then by Proposition~\ref{pr:ringpreimage}, $G$ is a $U$-factroid of $M$, where $M$ is thought of as an $A$-module by restriction of scalars along $\varphi$.  Let $j: L \into M$ be the ($A$-linear) inclusion map.  Then by Proposition~\ref{pr:modulepullback}, $G \cap L = j^{-1}(G)$ is a $U$-factroid of $L$ containing $S$.  Hence, $[S]^U_L \subseteq G \cap L \subseteq G$. Since $G$ was arbitrary, $[S]^U_L \subseteq [S]^T_M$.
\end{proof}

\begin{example}\label{ex:badringfactroidmaps}
Although factroids are well-behaved under inverse images of ring homomorphisms (see Proposition~\ref{pr:ringpreimage}(2)), they are not well-behaved under forward images.  That is, if  $\phi: A \ra B$ is a ring homomorphism, $W \subseteq A$ a multiplicative set, and $F$ a $W$-factroid of $A$, $\phi(F)$ is not necessarily a $\phi(W)$-factroid of $B$.

For a surjective example, let $A = k[x,y]$, $B = k[x]$, and $\phi:A \onto B$ the $k$-algebra map that sends $x\mapsto x$ and $y \mapsto 0$.  Let $W = k[x] \setminus \{0\} \subseteq A$.  Let $F$ be the vector space spanned by $1$ and $x^2 + y^3$.  Then by Example~\ref{ex:kxy}, $F$ is a $W$-factroid. But $\phi(F) = \spn_k\{1, x^2\}$ is not a $\phi(W)$-factroid of $B$, since $x^2 \in \phi(F)$ and $x\in \phi(W)$, but $x\notin \phi(F)$.

For an injective example, let $A = k[x,y]$, $B = k[x,y,z] / (x^2 - yz)$, and $\phi: A \ra B$ the obvious injective ring homomorphism.  Let $F = k[x] \subseteq A$ and $W = A \setminus \{0\} = \reg(A)$.  Then $F$ is a $W$-factroid of $A$, but $\phi(F)$ isn't a $\phi(W)$-factroid of $B$, since $yz =x^2\in \phi(F)$ and $y\in \phi(W)$ but $z\notin \phi(F)$.
\end{example}

The following proposition connects the $T$-factroids of $M$ to those of the $A$-submodules of $M$.

\begin{prop}\label{prop:subfactroids}
    Let $N$ be an $A$-submodule of $M$ and $F$ an additive subgroup of $N$.  Then one has $W_N(F) \cap W_M(N) \subseteq W_M(F) \subseteq W_N(F)$.  Equivalently: for any subset $T$ of $A$, if $F$ is a $T$-factroid of $M$, then $F$ is $T$-factroid of $N$; and if $F$ is a $T$-factroid of $N$ and $N$ is a $T$-factroid of $M$, then $F$ is a $T$-factroid of $M$. 
    \end{prop}

    \begin{proof}
        The proof is straightforward.
    \end{proof}

The following proposition describes the $T$-factroids of any quotient  of $M$. 

\begin{prop}\label{prop:quotientfactroids}
For any subset $T$ of $A$ and any $A$-submodule $N$ of $M$, the association $F \mapsto F/N$ defines a bijection from the $T$-factroids of $M$ containing $N$ to the $T$-factroids of the $A$-module $M/N$.  Consequently, one has $W_M(F) = W_{M/N}(F/N)$ for any additive subgroup $F$ of $M$ containing $N$.
\end{prop}

\begin{proof}
  The proof is straightforward.
\end{proof}

Since any $T$-factroid of $M$ contains $[0]^T_M$, one has the following corollary.

\begin{cor}\label{cor:factinjective}
For any subset $T$ of $A$, the association $F \mapsto F/[0]^T_M$ defines a bijection from the $T$-factroids of $M$ to the $T$-factroids of the $A$-module $M/[0]^T_M$.  Moreover, the $A$-module $M/[0]^T_M$ is the universal quotient $N$ of $M$ such that every element of $T$ is regular on $N$, or equivalently such that the zero submodule of $N$ is a $T$-factroid of $N$.
\end{cor}

As a consequence of Corollary~\ref{cor:factinjective}, given a subset $T$ of $A$, there is no harm in assuming that the zero submodule of $M$ is a $T$-factroid of $M$, or equivalently that every element of $T$ is regular on $M$.  Thus, there is no harm in assuming that the localization map $M \to \langle T \rangle^{-1}M$ is injective, i.e., that $M$ is (canonically isomorphic to) an $A$-submodule of $\langle T \rangle^{-1}M$.
This has some implications for Theorem~\ref{thm:recipG} and its corollaries.  For example, let $A$ be any commutative ring, $W$ a multiplicative set, and $M$ an $A$-module.  Set $\bar{M} := M/[0]^W_M$,  $\bar{S} := $ the image of the set $S$ under the projection $M \onto \bar M$, and $L := W^{-1}M = W^{-1}\bar M$.  Then by combining Corollary~\ref{cor:factinjective} with the last part of Theorem~\ref{thm:recipG}, we have $R^S_W(M) = \bigcup_{b \in W} ([b\bar S]^W_{\bar M} :_L b)$.

Finally, as a consequence of Propositions \ref{prop:subfactroids} and \ref{prop:quotientfactroids}, one has the following partial analogue of Proposition \ref{pr:modulepullback} for   images of $T$-factroids  under a module homomorphism.

\begin{cor}
    Let $T$ be a subset of $A$, let $h: M \to N$ be any homomorphism of $A$-modules, and let $F$ be an additive subgroup of $M$ containing $\ker h$.  Then  one has $W_{M}(F) \cap W_{N}(\im h) \subseteq W_N(h(F)) \subseteq W_{M}(F) = W_{\im h}(h(F))$.  Equivalently: if $F$ is a $T$-factroid of $M$ and $\im h$ is a $T$-factroid of $N$, then $h(F)$ is a $T$-factroid of $N$; and if $h(F)$ is a $T$-factroid of $N$, then $F$ is a $T$-factroid of $M$. 
\end{cor}

\begin{example}
Recall the surjective map from Example~\ref{ex:badringfactroidmaps}, where $A=k[x,y]$, $F = \spn_k\{x^2+y^3, 1\}$, $B = k[x]$, $T := k[x] \setminus \{0\} \subseteq A$, and $\phi: A \ra B$ is the $k$-algebra homomorphism with $x \mapsto x$ and $y \mapsto 0$.  Then setting $M := A$, and $N := B$ thought of as an $A$-module by restriction of scalars, this provides an example of a surjective $A$-linear map $h=\phi: M \ra N$ and a $W$-factroid $F$ of $M$ such that $h(F)$ is not a $T$-factroid of $N$.  Of course, $F$ does not contain $\ker h$.
\end{example}

\section{Finiteness in graded settings}\label{sec:graded}
 In this section, we show that for rings similar to polynomial rings over a field, factroids tend to be bounded in degree, and hence are finite-dimensional as vector spaces over the base field.

\begin{defn}
Recall that for an $\N$-graded or $\Z$-graded abelian group $(G,+)$, the \emph{degree} of a nonzero (but not necessarily homogeneous) element $g\in G$ can be defined as follows.  We have a unique decomposition $g = \sum_{i=c}^d g_i$, where $c\leq d$ in $\N$ (resp. $\Z$), each $g_i$ is in the homogeneous degree $i$ graded component of $G$, $g_c \neq 0$, and $g_d \neq 0$.  Then $\deg g := d$.
\end{defn}

\begin{lemma}[Guerrieri]\label{lem:degbound}
Let $A$ be an $\N$-graded ring, $W\subseteq A$ a multiplicative set that does not contain $0$,  and $M$ a torsion-free
$\Z$-graded $A$-module.  Let $S \subseteq M$ and $d\in \Z$ such that for all $s\in S$, $\deg s\leq d$.  Then for all $u \in [S]_M^W$, $\deg u\leq d$.
\end{lemma}

\begin{proof}
By Definition~\ref{cons:indfact} and induction, it suffices to show that $\deg u\leq d$ for all $u \in F_1^W(S)$.  For any $u\in F_1^W(S)$, there exist $u_i \in M$ and $r_i\in W$ such that $u = \sum_{i=1}^n u_i$ and $r_i u_i \in S$ for all $i$.  Since $r_i$ has nonnegative degree and degrees are additive (by torsion-freeness of $M$), for each $i$ we have $\deg(u_i) \leq \deg(r_i) + \deg(u_i) = \deg(r_i u_i) \leq d$.  Thus, $\deg(u) \leq \max_{1\leq i \leq n} \deg(u_i) \leq d$.
\end{proof}

\begin{thm}
Let $B$ be a Noetherian $\N$-graded domain over a field $B_0=k$, and let $A$ be a $k$-subalgebra of $B$.  Let $W \subseteq B$ be a multiplicative set and let $V = W \cap A$.  Let $M$ be a finitely generated torsion-free $\N$-graded $B$-module, and let $L$ be a nonzero $A$-submodule of $M$.  Then any finitely generated $V$-factroid $F$ of $L$ is finite-dimensional as a $k$-vector space, with every element of no bigger degree than any element of a factroid generating set of $F$.
\end{thm}

\begin{proof}
Let $S=\{s_1, \ldots, s_m\}$ be a generating set for the $V$-factroid $F$, so that $F=[s_1, \ldots, s_m]_L^V$.  For each $1\leq i \leq m$, let $d_i = \deg(s_i)$, and let $d = \max \{d_i \mid 1\leq i \leq m\}$.  
It follows from Lemma~\ref{lem:degbound} that every element of $[S]_M^W$ has degree $\leq d$. Since $M$ is a Noetherian graded $B$-module, every $M_i = \{x \in M \mid x \text{ is homogeneous of degree }i\}$ is a finitely generated $k$-vector space.  Since $[S]_M^W \subseteq \bigoplus_{i=0}^d M_i$, it follows that $[S]_M^W$ is a finite-dimensional $k$-vector space.  By Proposition~\ref{pr:VLinWM}, we have $F=[S]_L^V \subseteq [S]_M^W$, so that $F$ is also finite-dimensional over $k$, with all elements of degree $\leq d$.
\end{proof}

The above then applies when $B$ is a finitely generated $k$-algebra domain, and $A$ is a $k$-subalgebra of $B$ (we might even have $A=B$), where we set $M=B$ and $L=A$.  Hence, by finding linearly independent elements of $F$ over $k$ within the degree bound, one will eventually arrive at a finite vector-space basis of $F$ over $k$.

\section*{Acknowledgments}
First we would like to thank the anonymous referee, whose careful reading of and astute suggestions regarding the paper improved it substantially.  
Next, we thank Lorenzo Guerrieri for Example~\ref{ex:Gnotbracket}, Lemma~\ref{lem:degbound}, and the case of Lemma~\ref{lem:inthelp} where $M=A$ and $S = \{1\}$.  Moreover, he had the initial insight that led to the discovery of factroids.  Finally, we are grateful to the moderators of mathoverflow for creating the forum that paired us as coauthors.  Without that forum, this paper would not exist.

\providecommand{\bysame}{\leavevmode\hbox to3em{\hrulefill}\thinspace}
\providecommand{\MR}{\relax\ifhmode\unskip\space\fi MR }
\providecommand{\MRhref}[2]{%
  \href{http://www.ams.org/mathscinet-getitem?mr=#1}{#2}
}
\providecommand{\href}[2]{#2}

\end{document}